\documentclass[reqno]{amsart}
\usepackage{amsfonts}
\usepackage{amsmath}
\usepackage{amsthm}
\usepackage{amssymb}
\usepackage{amscd}
\usepackage{mathtools}
\usepackage{mathrsfs}
\usepackage[latin1,utf8]{inputenc}
\usepackage{fancyhdr}
\usepackage{cite}

\usepackage{graphicx}
\usepackage{float}
\usepackage[dvipsnames]{xcolor}
\usepackage{pgf,tikz,pgfplots}
\usetikzlibrary{arrows}
\pgfplotsset{compat=1.15}

\usepackage[numbers]{natbib}

\usepackage{hyperref}
\hypersetup{
    colorlinks=true,
    linkcolor=Violet,
    citecolor=Violet,
    urlcolor=Violet,
}

\makeatletter
\newcommand{\setword}[2]{%
	\phantomsection
	#1\def\@currentlabel{\unexpanded{#1}}\label{#2}%
}
\makeatother

\definecolor{ccqqww}{rgb}{0.8,0.,0.4}
\definecolor{ffccww}{rgb}{1.,0.8,0.4}
\definecolor{qqwwzz}{rgb}{0.,0.4,0.6}
\definecolor{ttzzqq}{rgb}{0.2,0.7,0.4}

\theoremstyle{plain}
\newtheorem{theorem}{Theorem}[section]
\newtheorem{lemma}[theorem]{Lemma}

\newtheorem{proposition*}{Proposition}
\newtheorem{corollary}[theorem]{Corollary}

\theoremstyle{definition}

\newtheorem{example}{Example}

\theoremstyle{remark}
\newtheorem{remark}{Remark}

\numberwithin{equation}{section}

\DeclareMathOperator{\R}{\mathbb{R}}

\DeclareMathOperator{\C}{\mathbb{C}}

\renewcommand{\H}{\mathbb{H}}
\renewcommand{\S}{\mathbb{S}}
\renewcommand{\a}{\mathfrak{a}}
\renewcommand{\k}{\mathfrak{k}}

\DeclareMathOperator{\I}{\mathcal{I}}

\DeclareMathOperator{\Z}{\mathbb{Z}}

\DeclareMathOperator{\g}{\mathfrak{g}}

\DeclareMathOperator{\m}{\mathfrak{m}}

\DeclareMathOperator{\End}{\text{End}}
\DeclareMathOperator{\Lie}{\text{Lie}}

\DeclareMathOperator{\Rep}{\text{Rep}}
\DeclareMathOperator{\id}{\text{id}}

\DeclareMathOperator{\Ad}{\text{Ad}}

\DeclareMathOperator{\rank}{\text{rank}}

\DeclareMathOperator{\Iso}{\text{Isom}}

\DeclareMathOperator{\Isot}{\text{Isot}}

\DeclareMathOperator{\res}{\text{res}}
\DeclareMathOperator{\sym}{\text{sym}}

\DeclareMathOperator{\U}{\mathcal{U}}

\DeclareMathOperator{\llangle}{\langle\langle}
\DeclareMathOperator{\rrangle}{\rangle\rangle}

\title[HIDDEN SYMMETRIES AND THE GENERIC SPECTRAL SETTING OF GENERALIZED LAPLACIANS ON HOMOGENEOUS SPACES]{HIDDEN SYMMETRIES AND THE GENERIC SPECTRAL SETTING OF GENERALIZED LAPLACIANS ON HOMOGENEOUS SPACES}

\author[D. S. de Oliveira \and
M. A. M. Marrocos
]
{Diego S. de Oliveira \and
Marcus A. M. Marrocos
}

\address{Centro de Matem\'atica, Computa\c{c}\~ao e Cogni\c{c}\~ao, Universidade Federal do ABC\\
Av. dos Estados, 5001\\
09210-580 Santo Andr\'e, S\~ao Paulo\\
Brazil.}
\email{diego.sousa@ufabc.edu.br}

\address{Departamento de Matemática, ICE, Universidade Federal do Amazonas\\
Av. General Rodrigo Octávio Jordão Ramos, 6200, Campus Universitário, Coroado I,\\
Manaus, AM, 69080-005.}
\email{marcusmarrocos@ufam.edu.br}
\thanks{The first author was supported in part by grant, Coordenação de Aperfeiçoamento de Pessoal de Nível Superior (CAPES). The  second author acknowledge support from Fundação de Amparo à Pesquisa do Estado do Amazonas (FAPEAM)}



\begin{document}

\maketitle



\pagestyle{fancy} 
\fancyhf[LHE]{\scalebox{0.6}{DIEGO S. DE OLIVEIRA AND MARCUS A. M. MARROCOS}}
\fancyhf[RHE]{\thepage}
\fancyhf[LHO]{\thepage}
\fancyhf[RHO]{\scalebox{0.55}{HIDDEN SYMMETRIES AND THE GENERIC SPECTRAL SETTING OF GENERALIZED LAPLACIANS ON HOMOGENEOUS SPACES}}
\fancyhf[CHE]{}
\fancyhf[CHO]{}

\fancyhf[LFE]{ }
\fancyhf[RFE]{ }
\fancyhf[LFO]{ }
\fancyhf[RFO]{ }
\fancyhf[CFE]{ }
\fancyhf[CFO]{ }



\begin{abstract}
The purpose of this work is to establish the spectral setting of some generalized Laplace operators associated to a generic $G$-invariant metric on a compact homogeneous space $M=G/K$. We show that this generic spectral configuration depends on the $G$-isometries and on some certain hidden symmetries constructed in the adjacent structures of $M$ and of these operators.
\end{abstract}

\setcounter{tocdepth}{1}
\tableofcontents

\section{Introduction}

\,

\,

\textbf{Why should we study Laplacians and similar operators?}

\,

The spectral theory of operators such as the Laplacian, the Hodge-Laplacian and other similar related operators (for example, Hamiltonian operators and the Dirac operator) has many applications including classical and quantum physical systems, the diffusion of heat on a given surface, the propagation of a biological virus or a tumor or even other diseases, the evolution of neural nets, data processing, and so on \cite{afrouzi2007population, zachmanoglou1986introduction, ranjan2012discrete, de2014laplacian}. 
Specially in physics, the aforementioned operators are related to dynamics and physical phenomenons in many areas as quantum physics, Newtonian mechanics, general relativity, astronomy, wave mechanics, fluid mechanics, string theory, just to cite a few of them \cite{moretti2017spectral, simon2007spectral, mantoiu2016spectral,helffer2013spectral}. In pure mathematics, these operators can give us information about the geometry and the topology of a manifold, for example.

\,

\textbf{Why should we study homogeneous spaces?}

\,

The study of homogeneous spaces is closely related to an important branch of mathematics known as Lie theory, which is connected to many other mathematical subjects as algebraic geometry, differential geometry, operator algebra, partial differential equations, number theory and mathematical physics \cite{baklouti2018geometric}. 

Homogeneous spaces form an important class of examples in the theory of manifolds. A homogeneous space is a $G$-manifold where the lie group $G$ acts transitively on $M$. In this case, $M$ can be identified as a coset manifold in the form $G/K$, where the subgroup $K$ can be taken as the isotropy subgroup $G$ of any fixed base point $p \in M$. These spaces generalize the class of Lie groups, which are prototypes to study symmetry and other geometrical or even topological transformations in a vast collection of models. Their origin is due to Sophus Lie, who was interested in the connection of the group theory and the solutions of differential equations in dynamics. Parallel to this development, Felix Klein applied the Lie groups as groups of transformations on a geometrical object and discussed the fundamental question of what a geometry is from that point of view \cite{Arvanitoyeorgos2003}. The Lie groups also generalize the matrix algebraic theory with ramifications in many computer science subjects, for example, beyond their wide range of interdisciplinary applications.   

 Homogeneous space also generalize the class of the symmetric spaces, which are important geometrical spaces that describe prototype models in the universe and in differential geometry. Specially in general relativity, it is common to study the Lorentzian symmetric spaces, which include examples as the Minkowski, De Sitter and anti-de Sitter spaces \cite{cahen1972lorentzian}. Also, the symmetric spaces are applied in some topics related to algorithms and numerical analysis such as polar type matrix decompositions, splitting methods for computation of the matrix exponential, composition of self adjoint numerical integrators and time symmetric dynamical systems \cite{munthe2001application}. Finally, we can also mention their importance to the random matrix theory with ramifications to disordered systems, quantum transport problems, the study of condensed matter physics and quantum chromodynamics, showing their extensive field of research issues \cite{Magnea2002introduction}.

\,

\textbf{Why should we study the spectra of differential operators?}

\,

The spectra of Laplacians give information about the spaces on which these operators are constructed and their adjacent structures. In many physical applications, the eigenvalues measure significant data as state levels of energy in quantum mechanics, special frequencies in wave mechanics and important scalar quantities in other physical systems. In the theory of manifolds, they not only provide geometrical information and harmonic analysis, but they also give us topological properties in the study of homology or cohomology, which measures `holes' and obstructions in the manifold. They also provide some other invariants in algebraic topology.  

The spectrum of a given Laplace operator also measures the presence of geometrical symmetries in some sense, that is, the presence of isometries in the given manifold. This remark is due to the fact that such operator commutes with isometries by the pullback action and, conversely, each diffeomorphism commuting with it is actually a isometry \cite{Kobayashi1963,watson1973manifold}. For example, if the spectrum of the Laplace-Beltrami operator is simple, then there is no significant presence of geometrical symmetries on the given manifold, as proved by Uhlenbeck \cite{Uhlenbeck1976} for a generic metric on a compact and connected Riemannian manifold. Here, the word \emph{generic} stands for the idea of most of metrics in the topological sense of second category sets of Baire, that is, there is a \emph{residual subset (the complement of a meager subset)} in the total space of Riemannian metrics such that, for each metric in this subset, the spectrum is simple. These considerations turn the presence of symmetries on the given manifold, endowed with an arbitrary metric, into a special rare behavior.

The previous paragraph is closely related to Hamiltonian operators in a quantum physical system, where the generic setting of the eigenspaces reflects the fact that the effective Hamiltonian has no degeneracy. On the other hand, the presence of symmetries is related to the existence of degeneracies in the system. Also, there are symmetries related to some non obvious degeneracies which are not apparent in the original geometric structure of the system in a first glance. These symmetries are referred as \emph{latent hidden symmetries} in the context of the article \cite{rontgen2021latent}. Degeneracies induced by latent symmetries can be preserved even under a breaking of the original rotational symmetry or other perturbations in the original geometric structure. The presence of degeneracies from obvious and non obvious symmetries can also decides the integrability of a dynamical physical system \cite{cariglia2014hidden}.

The analogous of Uhlenbeck's work for the Hodge-Laplace is not obvious and its generic panorama demands some extra hypothesis and adaptations. We mention some advances on closed manifolds of dimensions $3$ and $5$ considered by Enciso and Peralta-Salas  \cite{enciso2012nondegeneracy} e Gier \cite{gier2014eigenvalue}.

In the Laplacian's spectral problem, when the given manifold $M$ has a prescribed group of isometries $G$ acting on it, the eigenspaces must be $G$-modules for each $G$-invariant metric. Thus, the minimum expected panorama for the spectrum is that each eigenspace is an irreducible $G$-module at least, and the multiplicity of the corresponding eigenvalue depends on the degree of this irreducible $G$-module. 
 Therefore, in the restricted space of $G$-invariant metrics on $M$, we can not expect the spectrum to be simple in the generic setting. Instead, we can ask if the spectrum is (real or complex) $G$-simple, that is, if the eigenspaces are irreducible (real or complex) $G$-modules for a generic $G$-invariant metric. The $G$-modules organize the geometric symmetries given by the isometries in $G$. If there exists non obvious (hidden) transformations leaving the spectrum invariant, in the generic setting, this means that the spectrum could not be $G$-simple for a generic $G$-invariant metric. In the context of Hamiltonians on physical systems, a non generic $G$-simple spectrum could mean that, in most of cases, there are degeneracies which are not induced by the prescribed symmetry geometric group $G$. Therefore, the spectral problem could have important non obvious symmetries playing a essential role in the organization of the eigenspaces in the generic panorama. Degeneracies can be classified as normal or accidental as long the $G$-simplicity property holds or not \cite{wigner2012group}.

 Collecting up all the symmetries given by $G$ and other possible non obvious symmetries leaving the spectrum invariant, for a generic $G$-invariant metric, we are able to find a more general set of generalized symmetries $\widetilde{G}$ describing the spectral problem more suitably. The symmetries of $\widetilde{G}$ which are outside of $G$ can be interpreted as hidden symmetries for the spectrum in the generic panorama.

The question about the generic $G$-simple spectrum of the real Laplace-Beltrami operator associated to a $G$-manifold $M$ was formulated as a question inside the problem number 42 in the Yau's list of open problems \cite{yau1993open}. It is worth to mention that Yau, who formulated the problem a few decades ago, already was aware of possible non obvious transformations outside of the prescribed group of isometries $G$ commuting with the Laplace-Beltrami operator and leaving the eigenspaces invariant. 

Some advances were given since the formulation of the problem. Zelditch \cite{zelditch1990} proved that $G$-Manifolds, where $G$ is a finite group of isometries whose degrees of its irreducible representations are majored by the dimension of $M$, has the generic property of the $G$-simple spectrum. Results in the same direction can also be found in \cite{Marrocos2019, cianci2024spectral}, where the authors, in addition to other results, show that the eigenspaces are irreducible representations of the prescribed symmetry group $G$ for a generic $G$-invariant metric. Schueth \cite{Schueth2017} proved the $G$-simplicity of the spectrum for a generic left-invariant on some compact Lie groups $G$ involving products of $SU(2)$ and torus. After this, Petrecca and Röser \cite{Petrecca2019} proved the $G$-simplicity of the spectrum for a compact rank one symmetric space $M=G/K$, and they also proved that the same result is not true for irreducible symmetric spaces of rank $\geq 2$.

The non generic $G$-simplicity of the spectrum for symmetric spaces of rank $\geq 2$ opens a route of investigation in the search of possible non obvious (hidden) symmetries explaining the non generic $G$-simplicity feature of the spectrum for these highly symmetric manifolds and, more generally, for arbitrary homogeneous spaces.

Another route of investigation is open in the same Yau's problem for generalized Laplacians and other similar operators. We can mention some adjacent results. For the Hodge-Laplacian, for example, Ikeda and Taniguchi  \cite{Ikeda1978} already had established an algebraic machinery, in 1978, to deal with the eigenspaces by using representation theory, although they had other particular interests. Recently, Semmelmann and Weingart \cite{semmelmann2019standard} considered the class of the standard Laplace operators, and Casarino, Ciatti and Martini \cite{casarino2022weighted} studied Grushin operators.  In some sense, the operators of this article have similar features to them, specially on the abstract algebraic level of the representation theory. We will also show that they generalizes the Laplace-Beltrami operator and, for some parameters, they also generalizes the Hodge-Laplacian and Casimir operators.

\,

\textbf{How do we find hidden symmetries?}

\,

First, we need to understand what the word \emph{symmetry} represents to our purposes, since it is a broad concept applied to different contexts. Usually, symmetry in geometry is refereed as the isometry notion of preserving a geometric structure. Symmetry also could be a transformation preserving an algebraic property or even a property of a function, a map or an operator of interest. 

In our context, a symmetry is a map or transformation defined on the given manifold or on its adjacent structures that is responsible to leave the  eigenspaces of the given operator of interest invariant. For a $G$-invariant metric, the group $G$ is commonly a natural set of symmetries of geometric nature for the eigenspaces. In this context, we can regard the term \emph{hidden symmetry} as any of these symmetries which are non obvious or non apparent at a first glance, living outside of $G$, and yet playing a essential role in the description of the eigenspaces for a generic $G$-invariant metric.

How do we decide if there are hidden symmetries in the generic setting of a given spectral problem and how do we find them? In order to motivate and illustrate this question it is worth to spend some words about the $SO(4)$ symmetry group of the hydrogen atom which is related to the Kepler's 3D problem and to the Laplacian's spectral problem \cite[Ch.9]{singer2006linearity}. In the Kepler's 3D problem, for example, we expect at a first glance the conservation of certain levels of energy described only by $SO(3)$-symmetries, since $SO(3)$ is the natural connected group of isometries in the 3D euclidean space $\R^3$. But it can be shown that there are some non obvious 3D movements conserving these energy levels which correspond to $SO(4)$-isometries for the sphere $\mathbb{S}^3$ under the stereographic projection. So, the transformations of these non obvious movements can be considered hidden symmetries for this physical phenomenon.

The search of hidden symmetries depends on the context and specific features of the problem \cite{rontgen2021latent,cariglia2014hidden,singer2006linearity, Kobayashi2017}. 

\,

\textbf{The main problem}

\,

Let $M = G/K$ a compact and connected homogeneous space, where $G$ acts isometrically and transitively on $M$. 

What is the spectral setting for the spectrum of a Laplacian, or an operator with similar features, associated to a generic $G$-invariant metric on $M$? Does it hold the $G$-simplicity property? If not, can we find a generalized set of symmetries $\widetilde{G}$ containing $G$ and hidden symmetries outside of $G$ which determine the generic panorama for the eigenspaces?

\,

\textbf{The generalized Laplace operators considered in this article}

\,

Let $g$ a $(G\times K)$-invariant metric on $G$ that corresponds to a $G$-invariant metric on a compact connected homogeneous space $M=G/K$, and $(\tau,U)$ a $K$-representation $U$ with $K$-action $\tau$. For a $g$-orthonormal basis $\{Y_j \}$, the quadratic element $\sum_j Y_j^2$ belongs to the universal envelopment algebra $\U(\g^{\C})$.

We want to construct operators $\Delta_{g,U^*}$, on the $G$-module
\[
C^\infty(G,K;U^*) := \{ f \in C^{\infty}(G,U^*) \; | \; \forall x \in G, \, \forall k\in K, f(xk)=\tau^*(k^{-1})f(x)   \} \, ,
\]
by using induced actions of the left or the right regular representation on elements in the form $\sum_j Y_j^2 \in \U(\g^{\C})$. We also assume that $U^*$ has a real form $U^*_{\R}$ and the complex operator $\Delta_{g,U^*}$ has a corresponding real version $\Delta_{g,U^*_{\R}}$.

More specifically, in Section \ref{sec:generalized-laplacians}, we define, for a compact and connected Lie group $G$ with left-invariant metric $g$, the operators
\[
\Delta_{g,U^*}:= -\sum\limits_{j} R_*(Y_j)^2 :C^\infty(G,U^*) \to C^\infty(G;U^*) \, ,
\] 
\[
\Delta_{g,U^*_{\R}}:= -\sum\limits_{j} R_*(Y_j)^2 :C^\infty(G,U^*_{\R}) \to C^\infty(G;U^*_{\R}) \, ,
\] 
where $R_*: \U(\g^{\C}) \to \End\left(C^\infty(G;U^*)\right)$ is the extension to the universal enveloping algebra $\U(\g^{\C})$ of the induced Lie algebra representation associated to the right regular representation $R: G \to GL\left( C^\infty(G;U^*) \right)$, which is merely taken as the derivative of $R$ at the identity element of the group $G$. 

Similarly, in Section \ref{sec:generalized-laplacians-normal-homogeneous-spaces}, we define, for a compact and connected homogeneous space $M=G/K$ with a normal $G$-invariant metric $g$, the operators

\[
\Delta_{g,U^*}:= -\sum\limits_{j} L_*(Y_j)^2 :C^\infty(G,K;U^*) \to C^\infty(G,K;U^*) \, ,
\] 
\[
\Delta_{g,U^*_{\R}}:= -\sum\limits_{j} L_*(Y_j)^2 :C^\infty(G,K;U^*_{\R}) \to C^\infty(G,K;U^*_{\R}) \, ,
\] 
where $L_*$ is induced by the left regular representation, in a analogous way to the last paragraph.

In both cases, when we take $U^* = \C$, then $\Delta_{g,U^*}$ will coincide with the complex version of the Laplace-Beltrami operator, up to identifications. Similarly, $\Delta_{g,U^*_{\R}}$ will be the corresponding real version one.

For normal homogeneous spaces, we can consider the splitting $\g = \k \oplus \m$ induced by the quotient $M=G/K$, with isotropy representation $\m$, and the complexified $p$-exterior representation $U^*= \wedge^p \m^{*\C}$. In this case, $\Delta_{g,U^*}$ will be a Casimir Operator which coincides with the Hodge-Laplacian of $M$ acting on differential complex $p$-forms (up to an identification). Similarly, $\Delta_{g,U^*_{\R}}$ will be the real version of this Hodge-Laplacian \cite{Ikeda1978,Tsonev2018}.

For the sake of simplicity, from now on in this introduction, we will discuss the results and main ideas only for the complex operator $\Delta_{g,U^*}$, since the analogous analysis for the real operator $\Delta_{g,U^*_{\R}}$ are essentially the same up to some adaptations.

In order to verify that $\Delta_{g,U^*}:C^\infty(G,K;U^*) \to C^\infty(G,K;U^*)$ is well defined, we need necessarily that $\Delta_{g,U^*} f \in C^\infty(G,K;U^*)$ for any $f \in C^\infty(G,K;U^*)$. This necessary condition will represents the fact that $\Delta_{g,U^*}$ commutes with the $\tau^*(K)$-action on $U^*$ in some sense. So, the representation $U^*$ itself will produce some algebraic structural extra symmetries commuting with the operator $\Delta_{g,U^*}$. This should have an impact on the eigenspaces. 

So, the generic expected panorama in the space of $G$-invariant metrics is that the eigenspaces of $\Delta_{g,U^*}$ are described by the geometric prescribed symmetries $G$, the algebraic representation $U^*$ and possibly new hidden symmetries induced by the conditions inherent to the spectral problem in this context.

The main idea to establish such a generic panorama for the spectrum of $\Delta_{g,U^*}$ is to control the multiplicities of the eigenvalues in each irreducible submodule of $C^\infty(G,K;U^*)$. These multiplicities must be the minimum as possible. Beyond that, it is very important to guarantee that nonequivalent irreducible $G$-submodules of $C^\infty(G,K;U^*)$ do not share common eigenvalues, unless they are related by certain symmetries (apparent or not) in the adjacent structures of the objects in the spectral problem.

\,

\textbf{The first main result -- hidden symmetries for a generalized Laplacian's spectrum associated to left-invariant metrics on compact Lie groups}

\,

In Section \ref{sec:generalized-laplacians}, we prove Theorem \ref{teo:cap2-grupos-de-Lie-metricas-invariantes-a-esquerda}. Basically, it gives an algebraic criterium for a compact and connected Lie group which states that, when certain conditions are satisfied, we can not expect that the operator $\Delta_{g,U^*}$ has a $G$-simple spectrum for a generic left-invariant metric. Instead, we can construct a generalized group of symmetries $\widetilde{G}$ containing $G$ such that, under the conditions of the theorem, $\Delta_{g,U^*}$ has a $\widetilde{G}$-simple spectrum for a generic left-invariant metric. More specifically, if we put some inner product on $U^*$, allowing us to construct the orthogonal group $O(U^*)$, and if we define
\[
\mathbf{G} := \left\{ \begin{array}{l}
     \text{$G$, if $G$ admits only representations of real type,}  \\
     \text{$Q_8 \times G$, if $G$ admits any representation of complex or quaternionic type,}
\end{array} \right.
\]
where $Q_8$ is an isomorphic copy of the group of quaternions, then
the group $\widetilde{G}$ is given exactly by the product group $(O(U^*) \times \mathbf{G})$. The group $G$ organizes the $G$-isometries. The group $O(U^*)$ organizes the algebraic structural hidden symmetries, mentioned few paragraphs ago, for that kind of commutativity between $\Delta_{g,U^*}$ and the $\tau^*(K)$-action associated to the $K$-module $(\tau^*, U^*)$. Finally, the $Q_8$-group organizes some other algebraic hidden symmetries commuting with $\Delta_{g,U^*}$, which are induced by the quaternionic structural maps constructed from irreducible $G$-modules of complex or quaternionic type.

As an illustration, in Example \ref{ex:examples-Schueth}, we show that any group in the form 
\[
G = (SU(2) \times \cdots \times SU(2) \times T^n)/ \Gamma \, ,
\] 
where $\Gamma$ is a discrete central subgroup and $T^n$ is the $n$-torus, satisfies the conditions of Theorem \ref{teo:cap2-grupos-de-Lie-metricas-invariantes-a-esquerda} and, therefore, it satisfies the generic setting given by this result.

\,

\textbf{The second main result -- hidden symmetries in the spectral problem on compact normal homogeneous spaces}

\,

Backing to the compact and connected irreducible symmetric spaces or, more generally, normal homogeneous spaces in the form $M=G/K$, how do we find hidden symmetries playing role in the generic setting of the eigenspaces?

In fact, in Section \ref{sec:generalized-laplacians-normal-homogeneous-spaces}, we found certain combinatorial groups of non apparent symmetries in the root system lattice associated to $(G,K)$ impacting in the spectral description. In some way, this approach has similarities with that one studied in \cite{rontgen2021latent}, since the authors also construct latent hidden symmetries in certain lattices in their article.  These combinatorial groups have dislocated orthogonal transformations in the root system and the action of one of them puts the highest weights, associated to irreducible $G$-modules with the same Casimir eigenvalue, in the same orbit transitively, as we proved in Theorem \ref{th:transitive-action}. Thus, these dislocated orthogonal transformations are hidden symmetries preserving spectral properties.

There are also other hidden symmetries related to orthogonal groups constructed from the $K$-representation $U^*$ (see more on Subsection \ref{subsec:operators-on-normal-homogeneous-spaces}).

From Theorem \ref{th:transitive-action} and the idea on the last two paragraphs, we achieved some important consequences for the generic spectral setting of a compact and connected normal homogeneous space as the content of Theorems \ref{th:sec4-esps-homog-normais-mesmo-autovalor-casimir}
and \ref{th:sec4-esps-homog-normais-mesmo-autovalor-casimir-caso-real}. Theorem \ref{th:sec4-esps-homog-normais-mesmo-autovalor-casimir} studies the the complex version of the operator, while Theorem \ref{th:sec4-esps-homog-normais-mesmo-autovalor-casimir-caso-real} is the study of its real version.  They basically state that, for a generic normal metric, there is a  set $\widetilde{G}$ of generalized symmetries containing $G$, the hidden symmetries from the combinatorial groups in Theorem \ref{th:transitive-action} and the hidden symmetries associated to the orthogonal groups constructed from the $K$-module $U^*$, simultaneously. Examples \ref{ex:hodge-total-normal-homogeneous-spaces} and \ref{ex:hodge-total-rank1} are applications of these theorems to the spectral setting of the Hodge-Laplace operator associated to a generic normal metric. 

Illustrating the main ideas, in Example \ref{ex:hodge-total-rank1} we study the Hodge-Laplacian on compact rank one Lie groups. In this example, it is very clear which is the generalized symmetry group $\widetilde{G}$ containing $G$ and the important hidden symmetries which organize the eigenspaces in the main problem enunciated some paragraphs ago. It is proved that the group  $\widetilde{G} = O(U^*) \times (G \times G)$, where $O(U^*)$ is the orthogonal group of $U^*:= \bigoplus_p \wedge^p \g^{*\C}$ induced by an inner product, satisfies the $\widetilde{G}$-simplicity of the Hodge-Laplacian's spectrum for a generic bi-invariant metric. In this construction, $O(U^*)$ organizes the structural algebraic hidden symmetries in $\tau^*(K)$ which produces some kind of commutativity between $\Delta_{g,U^*}$ and the action of the $K$-module $(\tau^*,U^*)$ (we already discussed this process twice, previously in this introduction). On the other hand, the group $G \times G$ organizes the isometries related to the bi-invarance of the metric, that is, the first $G$-factor organizes the left-translations and the second $G$-factor organizes the right-translations (or vice versa).

\,

\textbf{The third main result -- generic estimate for a generalized Laplacian's spectrum associated to left-invariant metrics on compact Lie groups}

\,

Finally, in Section \ref{sec:back-to-Lie-groups-left-metrics}, we provide an interesting application on Lie groups with left-invariant metrics, by connecting the ideas of Sections \ref{sec:generalized-laplacians} and \ref{sec:generalized-laplacians-normal-homogeneous-spaces}. More specifically, in Corollary \ref{cor:cap2-grupos-de-Lie-metricas-invariantes-a-esquerda-aprimorado}, we show that each eigenspace of the given generalized Laplacian has some kind of upper quota. This means that there is a limited number of distinct irreducible $G$-modules inside of each eigenspace for a generic left-invariant metric $g$, and we describe which are these possible irreducible $G$-modules. This result is refereed as a generic estimate for the spectrum and it holds for arbitrary compact and connected Lie groups.




\,

\section{Preliminaries}

\,
Convention: all the manifolds $M$ and Lie groups, commonly denoted by $G$ and $K$, are assumed to be connected and compact throughout the whole article.

Unless we make other mentions, we consider $M = G/K$ as a compact connected $m$-dimensional homogeneous space with $G$ being a $N$-dimensional connected, compact and simple Lie group acting by left translations on $M$. $K$ is a 
 $k$-dimensional closed subgroup of $G$. We fix $g_0$ a bi-invariant metric on $G$ that induces on the Lie algebra level the splitting $\g = \k \oplus \m$, where $\m$ is the isotropy $K$-representation. We also fix
$(\tau, U)$ as a real or complex finite dimensional $K$-representation with action $\tau$ on the space $U$. The dual of $(\tau.U)$ is denoted by $(\tau^*,U^*)$. We take $g$ as an arbitrary $(G\times K)$-invariant metric on $G$ (i.e. invariant by $G$-left translations and $K$-right translations), which corresponds to a $G$-invariant metric on $M$. We remark that when the $G$-invariant metric $g$ on $M$ corresponds to a bi-invariant metric on $G$, then $(M,g)$ is called a normal homogeneous space with normal metric $g$. $\widehat{G}$ denotes a complete set of nonequivalent irreducible representations of the Lie group $G$. $\mathcal{U}(\g^{\C})$ is the universal enveloping algebra of $\g$.

We begin with some notation and basic definitions. First, we define the induced $G$-module of the $K$-module $U^*$ by
\[
C^\infty(G,K;U^*) := \{ f \in C^{\infty}(G,U^*) \; | \; \forall x \in G, \, \forall k\in K, f(xk)=\tau^*(k^{-1})f(x)   \} \, ,
\]
which is endowed with the left action
\[
L(x)f := f \circ \ell_{x^{-1}} \, , \; \forall x \in G, \; \forall f \in C^\infty(G,K;U^*),
\]
where $\ell_{x^{-1}}:G\ni y \mapsto x^{-1}y\in G$ is the left translation by $x^{-1}$. 

Now, we suppose that the real (resp. complex) representation $U^*$ is endowed with a $G$-invariant real (resp. hermitian) inner product $\langle \, \cdot \, , \, \cdot \, \rangle$ and we define
\[
\langle f_1 , f_2 \rangle_{L^2} := \int_G \langle f_1(x) , f_2(x) \rangle dx \, , \; \forall f_1,f_2 \in C^\infty(G,K;U^*) \, ,
\]
where $\int_G (\cdot) dx$ is the Haar integral of $G$ with unitary volume. So, we have the $L^2$-completion of $C^\infty(G,K;U^*)$ denoted by $L^2(G,K;U^*)$. This new space has a Peter-Weyl decomposition given by the Hilbert sum
\begin{equation} \label{eq:peter-weyl-general-decomposition}
    \bigoplus_{V^* \in \widehat{G}_{K,U^*}} \Isot(V^*) \simeq \bigoplus_{V^* \in \widehat{G}_{K,U^*}} (V \otimes U^*)^K \otimes V^* \, ,
\end{equation}
where $\widehat{G}_{K,U^*}:= \{ (\Pi^*,V^*)\in \widehat{G} \; | \, \dim (V\otimes U^*)^K>0 \}$ and $\Isot(V^*)$ denotes the $V^*$-isotypical component of $L^2(G,K;U^*)$ with respect to the $L$-action, endowed with an isomorphism
\begin{equation} \label{eq:isomorphisms-isotypical-general-decomposition}
    \begin{array}{c}
     \varphi_{V^*}: (V\otimes U^*)^K \otimes V^* \ni (v\otimes \omega ) \otimes \xi \mapsto \phi_{v,\omega,\xi} \in   \Isot(V^*) \; , \\
    \text{where $\phi_{v,\omega,\xi} := [(\Pi^*(\cdot)^{-1}\xi)(v)]\omega = [\xi(\Pi(\cdot)v)]\omega $ ,}
\end{array}
\end{equation}
defined by linear extension (for more details see \cite{Brocker1985}). 

Our main goal is to study operators $\Delta_{g,U^*} : L^2(G,K;U^*) \to L^2(G,K;U^*)$ commuting with the $(G\times K)$-isometries with respect to the metric $g$, by describing them inside each isotypal component in the decomposition \ref{eq:peter-weyl-general-decomposition} and in terms of the isomorphisms \ref{eq:isomorphisms-isotypical-general-decomposition}.

It is well known that for each complex finite dimensional $G$-module $V$, we have the $G$-isomorphism $V^* \simeq \overline{V}$, where $\overline{V}$ is the conjugate representation of $V$, that is, $\overline{V}$ coincides with $V$ as sets but any complex scalar $z$ acts as $\overline{z}$ in the scalar multiplication of $\overline{V}$.

A $G$-module $V$ is called of \emph{real type} if it admits a linear $G$-map $J_V:V \to \overline{V}$ such that $J_V^2 = \id$. $V$ is called of \emph{quaternionic type} if it admits a linear $G$-map $J_V:V \to \overline{V}$ such that $J_V^2 = -\id$. In both cases, $J_V$ is called a \emph{structural map}.  $V$ is called of \emph{complex type} if $V \ncong \overline{V}$ as $G$-modules, i.e. it is not of real type neither of quaternionic type. \footnote{Some authors prefer to define the structural map $J_V$ as a \emph{anti-linear} $G$-map in the form $J_V: V \to V$, instead of a \emph{linear} $G$-map $J_V:V \to \overline{V}$.}
For $V$ of complex type, we have that $V \oplus \overline{V} \simeq \H \otimes V$ is of quaternionic type (for more details, see \cite[Sec.II-6]{Brocker1985}).

The representations $V$ of quaternionic type are identified as representations in the category $\Rep(G,\H)$ (i.e. with scalars on $\H$) and, under this identification, the structural map $J_V$ is simply the left multiplication by the quaternionic element $j$. 

For each $V \in \widehat{G}$ there exists an irreducible real $G$-representation $V_{\R}$ such that 
\begin{equation} \label{eq:types-real} 
    \C \otimes V_{\R} \simeq \left\{ \begin{array}{l}
     \text{$V$, if $V$ is of real type,} \\
     \text{$\H\otimes V$, if $V$ is of complex or quaternionic type}
\end{array} \right. 
\end{equation}
(see \cite[Sec.1]{Petrecca2019} or \cite[Sec.2]{Schueth2017}).

Define $\Isot(V,\overline{V}) := \Isot(V)+\Isot(\overline{V}) \subset L^2(G,K;\C) $. Since $V \simeq \overline{V}$ for $V$ of real or quaternionic type, and $V \ncong \overline{V}$ for $V$ of complex type (as $G$-modules) then 

\begin{equation} \label{eq:isotipica}
\Isot(V,\overline{V}) = 
\begin{cases}
\Isot(V) + \Isot(\overline{V}) \text{, for $V$ of complex type,}\\
\Isot(V) \text{, for $V$ of real or quaternionic type.}
\end{cases}
\end{equation}
\,

Finally, we mention that when $K=\{e\}$ we have $C^\infty(G,K;U^*) = C^\infty(G;U^*)$ and we can also endow it with the right action
\[
R(x)f := f \circ r_{x} \, , \; \forall x \in G, \; \forall f \in C^\infty(G;U^*),
\]
where $r_{x}:G\ni y \mapsto yx\in G$ is the right translation by $x$. Therefore, each left construction on $C^\infty(G;U^*)$ has an analogous equivalent right construction and they are related by the inversion map on the Lie group $G$.

For details, see more on \cite{Brocker1985}.
\,


\subsection{Laplace operators} \label{preliminaries:laplace-operators}

We assume that $\{Y_j\}_{j=1}^N$ is a $g$-orthonormal basis of $\g$. For an arbitrary $G$-representation $(\Pi,V)$ (finite dimensional or not), with induced representation $\Pi_*:\U(\g^{\C}) \to \End(V)$, we define
\[
\Delta_g^V := -\sum\limits_{j=1}^N \Pi_*(Y_j^2) : V \to V \; .
\]
For a bi-invariant metric $g$, $\Delta_g^V$ is commonly referred as the Casimir element (or Casimir operator) of the representation $V$. We also define the operator
\begin{equation*}
    \Delta_g^{V^K} := -\sum\limits_{j=k+1}^N \Pi_*(Y_j^2) : V^K \to V^K \; ,
\end{equation*}    
which satisfies $\Delta_g^{V^K}= \Delta_g^V|_{V^K}: V^K \to V^K$.

\begin{remark}
If we follow the construction on \cite[Sec.1]{Petrecca2019}  by taking $V := C^\infty(G,K;\C)$, then we note that $\Delta_{g}^V$ coincides with the Laplace-Beltrami operator on $M$, up to an identification. Similarly, if we follow \cite[Secs.1-2]{Ikeda1978} by taking $g$ as a $(G\times G)$-invariant metric (a normal metric) and $V := C^\infty(G,K;\wedge^p \m^{*\C})$, then $\Delta_{g}^V$ is a Casimir operator which coincides (up to an identification) with the Hodge-Laplacian on $M$, acting on differential complex $p$-forms.
\end{remark}

\begin{remark} \label{remark:eigenspaces-j-invariant}
By \cite[Sec.1]{Petrecca2019}, if $V$ is a $G$-module of quaternionic type with structural map $J_V: V \to \overline{V}$, then the eigenspaces of $\Delta_g^V$ are $J_V$-invariant.
\end{remark}

\begin{remark} \label{remark:LaplacianV-LaplacianV*}
Also by \cite[Sec.1]{Petrecca2019}, for all $\xi \in V^*$, $\Delta_g^{V^*} \xi = \xi \circ \Delta_g^V $.
\end{remark}

\,


\subsection{Basic facts on root systems} \label{preliminaries:root-systems}

For a compact normal homogeneous space $(M,g)$, $M=G/K$, we know the Freudenthal's formula. That is, if we take $V^\mu \in \widehat{G}$ the irreducible representation with highest weight $\mu$ and $\delta$ the half sum of positive roots of the root system of $\g$, then 
\begin{equation} \label{eq:Freudenthal-formula}
    \begin{array}{l}
    \Delta_g^{V^\mu} = \lambda_\mu \id : V^\mu \to V^\mu \, , \\
      \text{where } \lambda_\mu := a_\mu^2 - g(\delta,\delta) \; ,  \\
     \text{and $a_\mu := g( \mu + \delta,\mu+\delta)^{1/2} $ .}
\end{array}
\end{equation}
The scalar $\lambda_{\mu}$ is called the Casimir eigenvalue of $V^\mu$ (or of $\mu$).

\begin{remark} \label{remark:same-casimir-eigenvalue}
    Note that $V^\mu, V^\eta \in \widehat{G}$ have the same Casimir eigenvalue if, and only if, $a_\mu = a = a_\eta$. The last claim is the same to say that $\mu$ and $\eta$ must be in the same sphere $\S_a(-\delta)$, with radius $a$ and centered in $-\delta$, in the root system of $\g$.
\end{remark}

\begin{remark} \label{remark:unitary-weight-vector}
  Suppose that the irreducible representation $V^\mu$ is endowed with a chosen $G$-invariant inner product and denote by $|| \cdot ||$ its induced norm. It is well known that the weight space associated to $V^\mu$ is an one-dimensional vector subspace of $V^\mu$. Therefore, in these conditions, there exists an unique weight vector $v_{\mu} \in V^\mu$ such that $||v_\mu|| = 1$, which we will refer as the \emph{unitary highest weight vector} of $V^\mu$.
\end{remark}

Let $(\a, R_{\g,\k})$ the root system associated to the pair $(G,K)$. Assume that $(\a, R_{\g,\k})$ is irreducible and that it is endowed with inner product $(\cdot,\cdot)$ (which is unique up to a scalar). The half sum of positive roots is $\delta$ and the Weyl group is $W$. We also fix a Weyl Chamber $C$ and define the lattices
\[
\Gamma_G := \{ H \in \a \, ; \; \exp_G(2\pi H) = e \}
\]
and
\[
\I := \sum\limits_{\alpha \in R_{\g,\k}} \Z \cdot \alpha^\vee \; ,
\]
where $\alpha^\vee := \frac{2}{\langle \alpha, \alpha \rangle } \alpha$. We recall that in the simply connected case, we have $\Gamma_G^* = \I^*$.

So, the spherical representations (i.e. the $G$-irreducible modules $V$ such that $\dim V^K >0$) are in one-to-one correspondence with the set $C \cap \Gamma_G^*$. 

For more details, see \cite[Chs. 8-10]{Hall2015}, \cite[Chs. VII and X]{Helgason2001} and \cite[Ch. II]{helgason2008geometric}.

\,


\subsection{Basic facts on the quaternion group $Q_8$} \label{preliminaries:Q8} 

It is well known that the finite quaternion group
\[
Q_8 := \{ \pm 1, \pm i, \pm j, \pm ij \}
\]
has only one $2$-dimensional complex irreducible representation which is isomorphic to faithful representation $\H$ (the ring of quaternions endowed with the standard $Q_8$-action by left multiplication). The other complex irreducible representations are of degree $1$.

\,


\subsection{The space of $G$-invariant metrics} \label{preliminaries:G-invariant-metrics}
We know that the $G$-invariant metrics on $M = G/K$ are in one-to-one correspondence with $(G \times K)$-invariant metrics on $G$. There are other characterizations that we will see now. First, let
\begin{equation*}
    \begin{array}{l}
        \sym_K(\m):=  \{ \kappa \in \End(\m) ; \; \text{$\kappa$ is symmetric and $\Ad_K$-equivariant}  \} \, , \\
        \sym_K^+(\m):= \{ \kappa \in \sym_K(\m)  ; \; \text{each eigenvalue of $\kappa$ is positive}  \} \, .
   \end{array}
\end{equation*}
Note that $\sym_K(\m)$ has an induced Euclidean topology from $\End(\m)$. The same is true for $\sym_K^+(\m)$.

For each $\kappa \in \sym_K^+(\m)$, we define 
\[
g_\kappa := g_0(\kappa^{-1}(\cdot) \, , \, \cdot):\m \times \m \to \R
\]
which corresponds to a $G$-invariant metric on $M$. Therefore, the set of $(G\times K)$-invariant metrics on $G$ can be identified as the space $\sym_K^+(\m)$, which turned out to be an open subset of $\sym_K(\m)$ (see more on \cite[Sec.1]{Petrecca2019}).

For our purposes, a residual subset in a topological space is a subset that can be expressed as a countable intersection of open dense subsets in this space. In this context,  the $G$-invariant metrics on $M$ associated to a residual subset of $\sym_K^+(\m)$ are called \emph{generic}. 

So, in this identification, the operators on Subsection \ref{preliminaries:laplace-operators} are actually indexed by a parameter $\kappa \in \sym_K^+(\m)$. We can generalize these constructions to operators with a parameter $\kappa \in \sym_K(\m)$ (a more general space). We do this in the following subsection. 

\,


\subsection{Operators on each representation}
\label{prelimiaries:operators-on-each-representation}
Fix a $g_0$-orthonormal basis $\{Y_j\}_{j=1}^N$ of $\g$ such that $\{Y_j\}_{j=k+1}^N$ is an orthonormal basis of $\m$ and let $(\rho,V) \in \widehat{G}$.

Now, we can define for each $\kappa = (\kappa_{ij}) \in \sym_K(\m)$ the operators
\begin{equation*}
    D^{V}(\kappa) := - \sum_{i,j = 1}^N \kappa_{ij} \, \rho_*(Y_i \cdot Y_j) : V \to V  \; .
\end{equation*}
Note that, when $k \in \sym_K^+(\m)$ corresponds to a metric $g$, then $D^V(\kappa) = \Delta_g^V$ and in this case we can write $ \Delta_\kappa^V := \Delta_g^V = D^V(\kappa)$. Therefore, the family of operators $\{ D^V(\kappa) \}_{\kappa \in \sym_K(\m)}$ is bigger than the family $\{ \Delta_\kappa^V \}_{\kappa \in \sym_K^+(\m)}$.

For each $\kappa \in \sym_K(\m)$, we can also define
\begin{equation*}
    D^{V^K}(\kappa) = -\sum\limits_{i, j=k+1}^N \kappa_{ij} \, \rho_*(Y_i \cdot Y_j) : V^K \to V^K \; .
\end{equation*}
Similarly, when $g$ corresponds to a $\kappa \in \sym_K^+(\m)$, we have $\Delta_{\kappa}^{V^K} = D^{V^K}(\kappa)$.

We note that $D^{V^K}(\kappa) = D^V(\kappa)|_{V^K}:V^K \to V^K$. For details on the precedent constructions, we refer to \cite[Sec.1]{Petrecca2019}.

\,


\subsection{Recent advances in the literature} \label{preliminaries:recent-advances-results-literature}
Now, we mention some important advances presented by \cite{Petrecca2019} and \cite{Schueth2017}.  Let $V, V_1,V_2 \in \widehat{G}$ and $\kappa \in \sym_K(\m)$.

There exists a map $\res:\C[t]\times \C[t] \to \C$, called \emph{resultant}, that satisfies for each pair of polynomials $p,q \in \C[t]$: (i) $\res$ is a polynomial map such that $\res(p,q)$ is a polynomial on the coefficients of $p$ and $q$; (ii) $\res(p,q)$ is the zero polynomial if and only if $p$ and $q$ share any common zero. This map is useful because it allow us to compare the eigenvalues of the operators $D^{V^K}(\kappa)$ (or $D^V(\kappa)$) by taking their corresponding characteristic polynomials. Define
\begin{equation*}
     \begin{array}{ccl}
        p_{V}(\kappa) &:=& \text{charac.polynomial}(D^{V^K}(\kappa)) \, ,\\
        a_{V_1,V_2} (\kappa) &:=& \res\left(p_{V_1}(\kappa), p_{V_2}(\kappa)\right) \, , \\
        b_{V}(\kappa) &:=& \res\left(p_{V}(\kappa), \frac{d}{dt} p_{V}(\kappa)\right) \, , \\
        c_{V}(\kappa) &:=& \res\left(p_{V}(\kappa), \frac{d^2}{dt^2} p_{V}(\kappa)\right)  .
    \end{array}
\end{equation*}
Thus we have polynomial maps
\[
a_{V_1,V_2},\,b_V,\,c_V: \sym_K(\m) \to \C
\]
and the following theorem:
\begin{theorem}
\small \label{teo:criterio-metrica-G-simples-generica} 
There exists a $G$-invariant metric $g$ such that the Laplace-Beltrami operator $\Delta_g$ on $M$ has a real $G$-simple spectrum if and only if the items bellow are simultaneously satisfied:

      \begin{enumerate}
          \item For all $V_1,V_2 \in \widehat{G}_K$, with $V_1 \ncong V_2$ and $V_1 \ncong V_2^*$, $a_{V_1,V_2}$ is not the zero polynomial.

          \item For all $V\in \widehat{G}_K$ of real or complex type, $b_{V}$ is not the zero polynomial.

          \item For all $V\in \widehat{G}_K$ of quaternionic type, $c_{V}$ is not the zero polynomial.
      \end{enumerate}
Moreover, the existence of a such metric is equivalent to say that the Laplace-Beltrami operator of a generic $G$-invariant metric on $M$ has real $G$-simple spectrum.
\end{theorem}

\begin{example} \label{ex:Schueth-SU2}
    Let $G = (SU(2) \times \cdots \times SU(2) \times T^n) / \Gamma $, where $\Gamma$ is a discrete central subgroup. Then for a generic left-invariant metric $g$ on $G$, the Laplace-Beltrami operator $\Delta_g$ is real $G$-simple. This result was proven by Schueth \cite{Schueth2017} in 2017.
\end{example}

When $M = G/K$ is a compact standard normal homogeneous space with metric $g$, then for each irreducible $G$-representation the operator $\Delta_g^V: V \to V$ is the Casimir element on $V$.  Thus, each operator $\Delta_g^{V^K} = \Delta_g^V|_{V^K}:V^K \to V^K$ is a multiple of the identity. In particular, for compact connected symmetric spaces, $\Delta_g^{V^K}$ has simple spectrum if and only if $\dim_{\C} V^K = 1$.

We say that $K$ is a \emph{spherical subgroup} of $G$ if, for each $V\in \widehat{G}_K$, we have that $\dim_{\C} V^K = 1$. Thus, as $K$ is a spherical subgroup, each operator $\Delta_g^{V^K}$ has a single eigenvalue and $\Isot(V^*) \simeq V^K \otimes V^* \simeq  \C \otimes V^* \simeq V^*$. So, the verification of the Theorem \ref{teo:criterio-metrica-G-simples-generica}, in this context, reduces to the verification of the item (1), that is, we only need to verify that distinct isotypical components do not produce common eigenvalues. It can be shown that for every irreducible symmetric pair $(G,K)$, the subgroup $K$ is spherical with respect to $G$. However we warn that not every subgroup $K$ of $G$ is spherical for an arbitrary homogeneous space $M=G/K$ (not even for a normal homogeneous space).

\begin{remark}\label{remark:real-or-complex-spherical-group} By \cite[Sec.1]{Petrecca2019}, when $K$ is a spherical subgroup of $G$, then the real version of the Laplace-Beltrami operator associated to a $G$-invariant metric $g$ on $G/K$ has a $G$-simple spectrum if and only if the complex version of Laplace-Beltrami operator associated to the metric $g$ has a $G$-simple spectrum. Therefore, in this context, there is no difference if we choose to study the $G$-simplicity of the real Laplacian or to study the $G$-simplicity of its complex version. 
\end{remark}

\begin{example} \label{ex:Petrecca-Symmetric-Spaces}
    Let $M$ a compact irreducible symmetric space. Then the Laplace-Beltrami operator is real $G$-simple if and only if $\rank(M) = 1$. Also, if $M$ is a product of compact rank one symmetric spaces, then the Laplace-Beltrami operator is real $G$-simple. In 2018, Petrecca and Röser \cite{Petrecca2019} argued that for $\rank(M) = 1$ there is only one spherical irreducible representation for each Casimir eigenvalue and that representation must be of real type, while for $\rank(M) \geq 2$ we can construct two or more non-equivalent and non-dual spherical representations with the same Casimir eigenvalue.
\end{example}

\,


\section{Generalized Laplacians on Lie groups with left-invariant metrics} \label{sec:generalized-laplacians}

\,

Let $M = G$ a compact connected Lie group endowed with a left-invariant metric $g$. In this section, we take $K =\{e\}$ and $\m = \g$.  Consider  $\{Y_j \}$ a $g$-orthonormal basis and $(\tau^*,U^*)$ a complex $K$-representation with a real form $U^*_{\R}$. Since $K=\{e\}$, then $U^*$ is merely an arbitrary finite-dimensional vector space.

We define
\[
\Delta_{g,U^*}:= -\sum\limits_{j} R_*(Y_j)^2 :C^\infty(G,U^*) \to C^\infty(G;U^*) \, ,
\] 
\[
\Delta_{g,U^*_{\R}}:= -\sum\limits_{j} R_*(Y_j)^2 :C^\infty(G,U^*_{\R}) \to C^\infty(G;U^*_{\R}) \, ,
\] 
where $R_*:\U( \g^{\C}) \to \End(C^\infty(G,K;U^*))$ is induced by the right regular representation, by taking its derivative on the identity element and its extension to the universal enveloping algebra $\U(\g^{\C})$.

Throughout this section, for each $V^* \in \widehat{G}$, we denote $\Isot(V^*)_\lambda$, $\Isot(V_{\R})_\lambda$ and $V_\lambda$ as the $\lambda$-eigenspaces of $\Delta_{g,U^*}|_{\Isot(V^*)}$, $\Delta_{g,U^*_{\R}}|_{\Isot(V_{\R})}$ and $\Delta_g^V$, respectively. Also, we put the set $\Isot(V,\overline{V})_\lambda$ as the $\lambda$-eigenspaces of $\Delta_{g,U^*}|_{\Isot(V, \overline{V})}$. Also, we recall that $V^* \simeq \overline{V}$ as $G$-modules.

For convenience of the representation theory, we begin studying the complex version $\Delta_{g,U^*}$ and, after this, we relate it to its corresponding real version $\Delta_{g,U^*_{\R}}$.


\,

\subsection{The operator $\Delta_{g,U^*}$ on each isotypical component of $L^2(G,U^*)$}

\,

\,

Of course, we can view $\Delta_{g,U^*}$ as an extended operator 
\[
\Delta_{g,U^*}:L^2(G,U^*) \to L^2(G,U^*) \, .
\]
By the isomorphisms in equation \ref{eq:isomorphisms-isotypical-general-decomposition}, we can  identify each isotypical component $\Isot(V^*)$ of $L^2(G;U^*)$ as $(V \otimes U^*)\otimes V^*$. We want to study the eigenspaces of $\Delta_{g,U^*}$ restricted to each of these isotypical components. Under the precedent identification, we have the following theorem: 

\begin{theorem} \label{th:sec3-laplacian-representative-functions} Let $G$ a compact and connected Lie group endowed with left-invariant metric $g$. Consider $U^*$ a finite-dimensional vector space. Then, the generalized Laplace operator $\Delta_{g,U^*}:L^2(G,U^*) \to L^2(G,U^*)$ can be identified as the operator
\[
\bigoplus_{V^* \in \widehat{G}}(\Delta_g^V \otimes \id)\otimes\id : \bigoplus_{V^* \in \widehat{G}} (V \otimes U^*) \otimes V^* \to \bigoplus_{V^* \in \widehat{G}} (V \otimes U^*) \otimes V^*  \, .
\]

In particular, the $\lambda$-eigenspace $\Isot(V^*)_\lambda$ of $\Delta_{g,U^*}|_{\Isot(V^*)}$ satisfies 
\[
\Isot(V^*)_\lambda \simeq (V_\lambda \otimes U^*)\otimes V^* \simeq U^* \otimes (V^*)^{\oplus \dim_{\C} V_\lambda } \; .
\]
\end{theorem}

\begin{proof}
Let $(v\otimes \omega)\otimes \xi \in (V\otimes U^*)\otimes V^* \simeq \Isot(V^*)$ and recall the isomorphism \ref{eq:isomorphisms-isotypical-general-decomposition}. It suffices to prove that
\[
\Delta_{g,U^*} \, \phi_{v,\omega,\xi} = \phi_{(\Delta_g^V v), \omega, \xi} \;.
\]
In fact, for any $x \in G$
\[
    \begin{array}{ccl}
       \Delta_{g,U^*} \, \phi_{v,\omega,\xi}\,(x)  &=& -\sum\limits_j  [R_*(Y_j)^2 \, \phi_{v,\omega,\xi}](x)  \\
         &=& -\sum\limits_j \left. \frac{d^2}{dtds}\right|_{s,t=0} \phi_{v,\omega,\xi} (\, x\exp((t+s)Y_j) \,)  \\
         &=& -\sum\limits_j \left. \frac{d^2}{dtds}\right|_{s,t=0} [\,\xi(\, \Pi(x) \Pi(\exp((t+s)Y_j)) v \,)\,]\omega \\
         &=& [\,\xi(\, \Pi(x)\Delta_g^V v \,)\,]\omega \\
         &=& [\,(\Pi^*(x)^{-1}\xi) (\Delta_g^V v)\,]\omega\\
         &=& \phi_{(\Delta_g^V v), \omega, \xi}\,(x) \, . 
    \end{array}
\]
\end{proof}


\subsection{Real $G$-properties vs. complex $(Q_8 \times G)$-properties} \label{subsec:real-vs-complex}

\,

\,

The Laplacians can be studied over $\R$ or $\C$ (in the last case, by simple arguments of complexification). For each real version of certain Laplace operator, the corresponding complex version can be considered itself a generalized Laplacian. The classical problems in analysis and differential geometry usually consider the operators over $\R$. But, when the problem is closely related to the representation theory, it is convenient to present them over $\C$, in order to make a better use of the algebraic machinery of this theory. For example, the Laplace-Beltrami operator $\Delta_g : L^2(M,\R) \to L^2(M,\R)$ has a complex version $\Delta_g : L^2(M,\C) \to L^2(M,\C)$. The Hodge-Laplacian on differential $p$-forms $\Delta_g : \Omega^p(M,\R) \to \Omega^p(M,\R)$ has a complex version $\Delta_g : \Omega^p(M,\C) \to \Omega^p(M,\C)$.  In general, the real version $\Delta_{g,U^*_{\R}}$ has corresponding complex version $\Delta_{g,U^*}$.

It is convenient to establish how to migrate from the constructions over $\R$ to the constructions over $\C$. It turned out that $G$-properties on the real versions of Laplacians have to be translated to $(Q_8 \times G)$-properties on the corresponding complex versions, whenever $\widehat{G}$ admits any representation of complex or quaternionic type. The main idea is that, when we migrate the real version of the operator to its complex version, we automatically earn some structural algebraic $Q_8$-maps that commute with it, where $Q_8$ is the finite group of quaternions.

Now, consider the following constructions and identifications:

\begin{itemize} \small
    \item[(J1)] For $V \in \widehat{G}$ of real type, the real structural map $j:= J_V$, which satisfies $J_V^2 = \id_{V}$, induces a map $j\otimes \id \otimes \id: V \otimes U^* \otimes \overline{V} \to V \otimes U^* \otimes \overline{V}$.

    \item[(J2)] For $V \in \widehat{G}$ of quaternionic type, the quaternionic structural map $j:=J_V$ satisfies $j^2 = -\id_{V}$ and it induces a map $j\otimes \id \otimes \id: V \otimes U^* \otimes \overline{V} \to V \otimes U^* \otimes \overline{V}$.

    \item[(J3)] For $V \in \widehat{G}$ of complex type, the $G$-module $(V \oplus \overline{V})$ admits a quaternionic structural map $j: (V \oplus \overline{V}) \to (V \oplus \overline{V})$, which induces a map  
    \[
    j\otimes \id \otimes \id: (V\otimes U^* \otimes \overline{V})\oplus(\overline{V} \otimes U^*  \otimes V) \to (V\otimes U^*\otimes \overline{V})\oplus(\overline{V}\otimes U^*\otimes V) \, .
    \]
\end{itemize}
The maps $j\otimes \id\otimes \id$ on (J1), (J2) e (J3) can be extended to a map  
\[
J : \bigoplus_{V \in \widehat{G}} V \otimes U^* \otimes \overline{V} \to \bigoplus_{V \in \widehat{G}} V \otimes U^* \otimes \overline{V} \, ,
\]
which commutes with the operator
\begin{equation}  \label{eq:laplacianV-tensor-id-tensor-id}
\bigoplus_{V \in \widehat{G}} \Delta_g^{V} \otimes \id\otimes \id : \bigoplus_{V \in \widehat{G}} V \otimes U^* \otimes \overline{V} \to \bigoplus_{V \in \widehat{G}} V \otimes U^* \otimes \overline{V} \; ,
\end{equation}
(this commutativity follows from Remark \ref{remark:eigenspaces-j-invariant}). Therefore $J$ commutes with $\Delta_{g,U^*}$. Moreover, $\Delta_{g,U^*}$ commutes with the $Q_8$-group given by the following group of automorphisms on $L^2(G,K;U^*)$:
\begin{equation} \label{eq:Q8-group}
    Q_8 := \{ \pm \id, \pm i \id, \pm J, \pm iJ \} \, .
\end{equation}

\begin{remark} \label{remark:translating-G-properties-to-Q8xG-properties}
  We recall that irreducible representations of quaternionic or complex type $V$ satisfy that $V\oplus \overline{V}$ is the complexification of an unique real irreducible representation $V_{\R}$. Thus, the analysis of the complex version of $\Delta_{g,U^*}$ on $\Isot(V,\overline{V})$ reduces to the analysis of $\Isot(V_{\R})$ on its corresponding real version $\Delta_{g,U^*_{\R}}$. Because of that, we do not see any effect of the $Q_8$-group given by equation \ref{eq:Q8-group} in the real version of the operator, even in the presence of the quaternionic or complex type representations. If  $\widehat{G}$ has only representations of real type, it is irrelevant construct the $Q_8$-group since that the correspondence passage from real type representations to real $G$-modules is well-behaved and it is not demands any extra analysis.
\end{remark}

The last remark motivates us to consider the following notation:

\[
\mathbf{G} := \left\{ \begin{array}{l}
     \text{$G$, if $\widehat{G}$ has only representations of real type,}  \\
     \text{$Q_8 \times G$, if $\widehat{G}$ has any representation of complex or quaternionic type.}
\end{array} \right. 
\]

Assume that $U^*$ is endowed with an real inner product. Then, we can consider the orthogonal group $O(U^*)$. The standard representation $U^*$ of $O(U^*)$ is clearly irreducible. Similarly, for the real space $U^*_{\R}$, we have the irreducible standard representation $U^*_{\R}$ of the orthogonal group $O(U^*_{\R})$. Also, if the linear action of a group $G_j$ is irreducible over $V_j$, for $j=1,2$, then $V_1 \otimes V_2$ is an irreducible $(G_1 \times G_2)$-module, with the tensor action.

Now, we have three corollaries of Theorem \ref{th:sec3-laplacian-representative-functions} as bellow. The first one is:

\begin{corollary} \label{cor:complex-type}
Let $G$ a compact and connected Lie group endowed with left-invariant metric $g$. Consider $U^*$ a finite-dimensional vector space. For each $V \in \widehat{G}$ of complex type and $\lambda$ an eigenvalue of $\Delta_g^V$ with multiplicity $m(\lambda)$ (therefore $\lambda$ also has multiplicity $m(\lambda)$ for the operator $\Delta_g^{V^*}$), it holds the isomorphisms 
\[
\Isot(V,\overline{V})_\lambda \simeq U^*\otimes (V\oplus \overline{V})^{\oplus m(\lambda)} \simeq U^*\otimes (\H \otimes V)^{\oplus m(\lambda)} \, .
\]
Moreover, the operator $\Delta_{g,U^*}|_{\Isot(V,V^*)}$ has $(\,O(U^*) \times \mathbf{G}\,)$-simple spectrum if and only if $\Delta_g^V$ has simple spectrum.
\end{corollary}

 We know that for $V \in \widehat{G}$ of quaternionic type, the eigenvalues of the operator $\Delta_g^V$ have even multiplicities (because the eigenspaces are $J_V$-invariant for the structural quaternionic map $J_V$). Therefore, the minimum multiplicity for each eigenvalue of $\Delta_g^V$ is $2$. Thus, we have the second corollary:

\begin{corollary}  \label{cor:quaternionic-type}
Let $G$ a compact and connected Lie group endowed with left-invariant metric $g$. Consider $U^*$ a finite-dimensional vector space. For each $V \in \widehat{G}$ of quaternionic type and $\lambda$ an eigenvalue of $\Delta_g^V$ with multiplicity $m(\lambda)$, it holds $\Isot(V,\overline{V}) = \Isot(\overline{V}) = \Isot(V) $ and the isomorphisms
\[
\Isot(V,\overline{V})_\lambda \simeq U^*\otimes V^{\oplus m(\lambda)} \simeq U^*\otimes (V\oplus \overline{V})^{\oplus m(\lambda)/2} \simeq U^*\otimes (\H \otimes V)^{\oplus m(\lambda)/2} \, .
\]
Moreover, the operator $\Delta_{g,U^*}|_{\Isot(V,V^*)}$ has $(\,O(U^*) \times \mathbf{G}\,)$-simple spectrum if and only if each eigenvalue of $\Delta_g^V$ has multiplicity $2$.
\end{corollary}

Finally, for representations of real type, we have the third corollary:

\begin{corollary}  \label{cor:real-type}
Let $G$ a compact and connected Lie group endowed with left-invariant metric $g$. Consider $U^*$ a finite-dimensional vector space. For each $V \in \widehat{G}$ of real type and $\lambda$ an eigenvalue of $\Delta_g^V$ with multiplicity $m(\lambda)$, it holds $\Isot(V,\overline{V}) =\Isot(V) = \Isot(\overline{V})$ and the isomorphisms
\[
\Isot(V,V^*)_\lambda \simeq U^*\otimes V^{\oplus m(\lambda)} \simeq U^*\otimes (\,\C \otimes V^{\oplus m(\lambda)} \,) 
\]
Moreover, the operator $\Delta_{g,U^*}|_{\Isot(V,V^*)}$ has $(\,O(U^*) \times \mathbf{G}\,)$-simple spectrum if and only if $\Delta_g^V$ has simple spectrum.
\end{corollary}

\,


\subsection{The main criterium of the section}

\,

\,

We recall that we are in the case of $K=\{e\}$ and $\m = \g$, in this whole section. According to Subsection \ref{preliminaries:recent-advances-results-literature}, we can consider the polynomial maps 
\[
a_{V_1,V_2},b_V,c_V:\sym_K(\m) \to \C \, .
\]

Now, we state the main criterium in this section.

\,

\begin{theorem}  \label{teo:cap2-grupos-de-Lie-metricas-invariantes-a-esquerda}
    A left-invariant metric $g$ on a compact and connected Lie group $G$ satisfies that the operator $\Delta_{g,U^*_{\R}}$ has a real $(O(U^*_{\R}) \times G)$-simple spectrum if and only if the operator $\Delta_{g,U^*}$ has a complex $(O(U^*) \times \mathbf{G})$-simple spectrum. Beyond that, the existence of a such metric is equivalent to the following items simultaneously satisfied:

         \begin{enumerate}
          \item For all $V',V \in \widehat{G}$, with $V' \ncong V, V^*$,  $a_{V',V}$ does not vanish identically.

          \item For all $V\in \widehat{G}$ of real or complex type, $b_{V}$ does not vanish identically.

          \item For all $V\in \widehat{G}$, of quaternionic type, $c_{V}$ does not vanish identically.
      \end{enumerate}
Moreover, such metric, when its existence holds, is generic on the space of the left-invariant metrics of $G$.        
\end{theorem}

\begin{proof} \,

    \vspace{0.5cm}
    \emph{First part: $\Delta_{g,U^*_{\R}}$ is $(O(U^*_{\R}) \times G)$-simple iff $\Delta_{g,U^*}$ is $(O(U^*) \times \mathbf{G})$-simple.}
    \vspace{0.2cm}

    Let $g$ a left-invariant metric on $G$. Take $\lambda$ as an eigenvalue of $\Delta_g^V$ with multiplicity $m(\lambda)$. By Subsection \ref{subsec:real-vs-complex}, we have

    \begin{equation*} \tag{I}
    \Isot(V,\overline{V})_{\lambda} \simeq \left\{ \begin{array}{l}
     \text{$U^* \otimes V^{\oplus m(\lambda)} \simeq U^* \otimes (\C \otimes V^{\oplus m(\lambda)})$, if $V$ is of real type,} \\
     \text{$U^* \otimes(\H\otimes V)^{\oplus m(\lambda)}$, if $V$ is of complex type,} \\ \text{$U^* \otimes(\H\otimes V)^{\oplus \,m(\lambda)/2}$, if $V$ is of quaternionic type}
\end{array} \right.  
\end{equation*}
(note that $(\H \otimes V)^{\oplus\, m(\lambda)/2} \simeq V^{\oplus m(\lambda)}$ for $V$ of quaternionic type).

Take $V_{\R}$, as in equation \ref{eq:types-real}, an irreducible real representation such that  
\begin{equation*}
    \C \otimes V_{\R} \simeq \left\{ \begin{array}{l}
     \text{$V$, if $V$ is of real type,} \\
     \text{$\H\otimes V$, if $V$ is of complex or quaternionic type.}
\end{array} \right.
\end{equation*}
Thus,
\begin{equation*}
    \Isot(V,\overline{V})_\lambda \simeq \left\{ \begin{array}{l}
     \text{$U^* \otimes (\C \otimes V_{\R})^{\oplus m(\lambda)}$, if $V$ is of real type,} \\
     \text{$U^* \otimes(\C \otimes V_{\R})^{\oplus m(\lambda)}$, if $V$ is of complex type,} \\ \text{$U^* \otimes(\C \otimes V_{\R})^{\oplus \,m(\lambda)/2}$, if $V$ is of quaternionic type,}
\end{array} \right. 
\end{equation*}
or equivalently
\begin{equation*}
    \Isot(V,\overline{V})_\lambda \simeq \left\{ \begin{array}{l}
     \text{$(\C \otimes V_{\R})^{\oplus m(\lambda) \dim_{\C} U^*}$, if $V$ is of real type,} \\
     \text{$(\C \otimes V_{\R})^{\oplus m(\lambda)\dim_{\C} U^*}$, if $V$ is of complex type,} \\ \text{$(\C \otimes V_{\R})^{\oplus \,(m(\lambda)/2)\dim_{\C} U^*}$, if $V$ is of quaternionic type.}
\end{array} \right. 
\end{equation*}

We know that $U^* = \C \otimes U^*_{\R}$. Since $\dim_{\C} U^* = \dim_{\R} U^*_{\R}$, then the $\lambda$-eigenspace of $\Delta_{g,U^*_{\R}}$, restricted to $\Isot(V,\overline{V})\cap C^\infty(G,K;U^*_{\R})$ is given by

\begin{equation*}
     \Isot(V_{\R})_\lambda \simeq \left\{ \begin{array}{l}
     \text{$V_{\R}^{\oplus m(\lambda)\dim_{\R} U^*_{\R}}$, if $V$ is of real type,} \\
     \text{$V_{\R}^{\oplus m(\lambda)\dim_{\R} U^*_{\R}}$, if $V$ is of complex type,} \\ \text{$ V_{\R}^{\oplus \,(m(\lambda)/2)\dim_{\R} U^*_{\R}}$, if $V$ is of quaternionic type,}
\end{array} \right. 
\end{equation*}
or equivalently
\begin{equation*}   \tag{II}
    \Isot(V_{\R})_\lambda \simeq \left\{ \begin{array}{l}
     \text{$U^*_{\R} \otimes_{\R} V_{\R}^{\oplus m(\lambda)}$, if $V$ is of real type,} \\
     \text{$U^*_{\R} \otimes_{\R} V_{\R}^{\oplus m(\lambda)}$, if $V$ is of complex type,} \\ \text{$U^*_{\R} \otimes_{\R} V_{\R}^{\oplus \,m(\lambda)/2}$, if $V$ is of quaternionic type.}
\end{array} \right. 
\end{equation*}

By (I) and (II), if $V$ is of real or complex type, then $\Delta_{g,U^*}|_{\Isot(V,\overline{V})}$ is complex $(\,O(U^*)\times \mathbf{G}\,)$-simple if and only if $\Delta_g^V$ has simple spectrum if and only if $\Delta_{g,U^*_{\R}} |_{\Isot(V_{\R})}$ is real $(O(U^*_{\R})\times G)$-simple. Similarly, if $V$ is of quaternionic type, then $\Delta_{g,U^*}|_{\Isot(V,\overline{V})}$ is complex $(\,O(U^*)\times \mathbf{G}\,)$-simple if and only if each eigenvalue of $\Delta_g^V$ has multiplicity $2$ if and only if $\Delta_{g,U^*_{\R}} |_{\Isot(V_{\R})}$ is real $(O(U^*_{\R})\times G)$-simple. By collecting all isotypical components together, we conclude that $\Delta_{g,U^*} $ is complex $(\,O(U^*)\times \mathbf{G}\,)$-simple if and only if $\Delta_{g,U^*_{\R}}$ is real $(O(U^*_{\R})\times G)$-simple.

\vspace{0.5cm}
    \emph{Second part: There exists a left-invariant metric $g$ such that the operator $\Delta_{g,U^*}$ is $(O(U^*) \times \mathbf{G})$-simple iff it holds (1), (2) and (3). Moreover, this existence is generic in the space of left-invariant metrics}
    \vspace{0.2cm}

$(\Rightarrow)$ Assume that $g$ is a left-invariant metric such that the operator $\Delta_{g,U^*}$ is $(O(U^*) \times \mathbf{G})$-simple. By following the first part, we conclude that for any $V^* \in \widehat{G}$ of real or complex type, the operator $\Delta_g^V$ has simple spectrum, that is, $b_V(g) \neq 0$. Similarly, for any $V^* \in \widehat{G}$ of quaternionic type, each eigenvalue of $\Delta_g^V$ has multiplicity $2$, meaning that $c_V(g) \neq 0$. Now suppose that there are $V,V' \in \widehat{G}$, with $V' \ncong V,V^*$ such that $a_{V',V}(g) = 0$. Then, there is some common eigenvalue $\lambda$ for the operators $\Delta_g^{V}$ and $\Delta_g^{V'}$. Therefore, the $\lambda$-eigenspace of $\Delta_{g,U^*}$ is a direct sum of at least two $(O(U^*) \times \mathbf{G})$-modules (one from $\Delta_g|_{\Isot(V,V^*)}$ and other from $\Delta_g|_{\Isot(V',{V'}^*)}$), which contradicts the $(O(U^*) \times \mathbf{G})$-simplicity of $\Delta_{g,U^*}$. Thus, $a_{V',V}(g) \neq 0$. So, it holds (1), (2) and (3).

$(\Leftarrow)$ Assume (1), (2) and (3). Let $V, V' \in \widehat{G}$ such that $V' \ncong V,V^*$. Since the polynomial $a_{V',V}$ does not vanishes identically on $\sym(\g)$, then it not vanishes identically on its open subset $\sym^+(\g)$. Let $Z_1$ the union of the zeros of all the maps $a_{V',V}|_{\sym^+(\g)}$, with irreducible $G$-modules $V' \ncong V,V^*$. Since $Z_1$ is a family of countable zeros of polynomials defined on the open set $\sym^+(\g)$, then $Z_1$ is a meager set on $\sym^+(\g)$. Similarly, we can define the meager set $Z_2 \subset \sym^+(\g)$ as the countable union of zeros of all the maps $b_V$, with $V \in \widehat{G}$ of real or complex type, and the meager set $Z_3 \subset \sym^+(\g)$ as the countable union of zeros of all the maps $c_V$, with $V \in \widehat{G}$ of quaternionic type. Therefore, $Z_1 \cup Z_2 \cup Z_3$ is a meager subset of $\sym^+(\g)$ or, equivalently, $\sym^+(\g) - (Z_1 \cup Z_2 \cup Z_3)$ is a residual subset of $\sym^+(\g)$. Note that, for any generic metric $g \in \sym^+(\g) - (Z_1 \cup Z_2 \cup Z_3)$, we have by construction: 
    \begin{enumerate}
        \item $a_{V,V'}(g) \neq 0$ for $V, V' \in \widehat{G}$ such that $V' \ncong V,V^*$,

        \item $b_V(g) \neq 0$, for any $V \in \widehat{G}$ of real or complex type,

        \item $c_V(g) \neq 0$, for any $V \in \widehat{G}$ of quaternionic type.
    \end{enumerate}
 They imply that
 \begin{enumerate}
        \item[(i)] $\Delta_g^{V'}$ and $\Delta_g^{V}$ do not share any common eigenvalue for $V, V' \in \widehat{G}$ such that $V' \ncong V,V^*$,

        \item[(ii)] $\Delta_g^{V}$ has simple spectrum for any $V \in \widehat{G}$ of real or complex type, 

        \item[(iii)] each eigenvalue of $\Delta_g^{V}$ has multiplicity $2$ for any $V \in \widehat{G}$ of quaternionic type.
    \end{enumerate}
 These three conditions collected up imply that $\Delta_{g,U^*}$ has a $(O(U^*) \times \mathbf{G})$-simple spectrum, by the first part of this proof.
\end{proof}

\begin{example} \label{ex:examples-Schueth}
Let $G = (SU(2) \times \cdots \times SU(2) \times T^n)/ \Gamma$, where $\Gamma$ is a discrete central subgroup. By Theorem \ref{teo:criterio-metrica-G-simples-generica} and Example \ref{ex:Schueth-SU2}, the items (1), (2) and (3) in the Theorem \ref{teo:cap2-grupos-de-Lie-metricas-invariantes-a-esquerda} are simultaneously satisfied. Then we conclude that $\Delta_{g,U^*_{\R}}$ and $\Delta_{g,U^*}$ has real $(O(U^*_{\R})\times G)$-simple and complex $(O(U^*)\times \mathbf{G})$-simple spectrum, respectively. In particular,

    \begin{enumerate}

        \item[(a)] Let $U^*_{\R} = \R$ and $U^* = \C$. The real version of the Laplace-Beltrami operator $\Delta_{g,\R}$ has a real $G$-simple spectrum for a generic left-invariant metric. Similarly, the complex version of the Laplace-Beltrami operator $\Delta_{g,\C}$ has a complex $\mathbf{G}$-simple spectrum for a generic left-invariant metric.  
        
        \item[(b)] Let $U^*_{\R} = \wedge^p \g^*$ and $U^* = \wedge^p \g^{*\C}$. We know that for bi-invariant metrics $g_0$, $\Delta_{g_0,U^*}$ is the Hodge-Laplacian on complex $p$-forms. For a generic left-invariant metric $g$,  $\Delta_{g,U^*}$ has a complex $(O(U^*)\times \mathbf{G})$-simple spectrum. Similarly, for a generic left-invariant metric $g$,  $\Delta_{g,U^*_{\R}}$ has a real $(O(U^*_{\R})\times G)$-simple spectrum. The same argument applies for the case in which we take $U^*_{\R} = \bigoplus\limits_p \wedge^p \g^*$ and $U^* = \bigoplus\limits_p \wedge^p \g^{*\C}$.
        
    \end{enumerate}
   
\end{example}

\,


\section{Generalized Laplacians on normal homogeneous spaces} \label{sec:generalized-laplacians-normal-homogeneous-spaces}

\,

\,

Let $g$ a $(G\times G)$-invariant metric on $G$ that corresponds to a $G$-invariant normal metric on a compact and connected homogeneous space $M=G/K$. Consider  $\{Y_j \}$ a $g$-orthonormal basis and $(\tau^*,U^*)$ a complex $K$-representation with real form $U^*_{\R}$.

For $(M,g)$, we can consider the operator
\[
\Delta_{g,U^*}:= -\sum\limits_{j} L_*(Y_j)^2 :C^\infty(G,K;U^*) \to C^\infty(G,K;U^*) \, ,
\] 
where 
\[
C^\infty(G,K;U^*) := \{ f \in C^{\infty}(G,U^*) \; | \; \forall x \in G, \, \forall k\in K, f(xk)=\tau^*(k^{-1})f(x)   \} \, 
\]
and $L_*:\U( \g^{\C}) \to \End(C^\infty(G,K;U^*))$ is induced by the left regular representation. We also assume that $U^*$ has a real form $U^*_{\R}$ and we define
\[
\Delta_{g,U^*_{\R}}:= -\sum\limits_{j} L_*(Y_j)^2 :C^\infty(G,K;U^*_{\R}) \to C^\infty(G,K;U^*_{\R}) \, .
\] 

Of course, if we take $U^*_{\R} = \R$ and $U^* = \C$, then $\Delta_{g,U^*_{\R}}$ and $\Delta_{g,U^*}$  are the real and the complex version of the Laplace-Beltrami operator, respectively. 

The pair $(G,K)$ induces the splitting $\g = \k \oplus \m$, where $\m$ is the isotropy representation. If we take the complexified $p$-exterior representation $U^*= \wedge^p \m^{*\C}$, then $\Delta_{g,U^*}$ coincides with the Hodge-Laplacian of $M$ acting on the differential complex $p$-forms of $M$ (up to identifications).

By the isomorphisms in equation \ref{eq:isomorphisms-isotypical-general-decomposition}, we can  identify each isotypical component $\Isot(V^*)$ of $L^2(G,K;U^*)$ as $(V \otimes U^*)^K\otimes V^*$. We want to study the eigenspaces of $\Delta_{g,U^*}$ restricted to each of these isotypical components, under the precedent identification, as it follows in the next theorem: 

\,

\begin{theorem} \label{th:sec4-laplacian-representative-functions-normal-homogeneous-spaces} Let $M=G/K$ a compact and connected homogeneous space, endowed with a normal metric $g$ which corresponds to a $(G\times G)$-invariant metric on $G$. Then, the generalized Laplace operator
$\Delta_{g,U^*}:L^2(G,K;U^*) \to L^2(G,K;U^*)$ can be identified as the operator
\[
\bigoplus_{V^* \in \widehat{G}_{K,U^*}} (\id \otimes \id)\otimes \Delta_g^{V^*}  : \bigoplus_{V^* \in \widehat{G}_{K,U^*}} (V \otimes U^*)^K \otimes V^* \to \bigoplus_{V^* \in \widehat{G}_{K,U^*}} (V \otimes U^*)^K \otimes V^*  \; 
\]
or the operator
\[
\bigoplus_{V^* \in \widehat{G}_{K,U^*}} (\Delta_g^V \otimes \id)\otimes\id : \bigoplus_{V^* \in \widehat{G}_{K,U^*}} (V \otimes U^*)^K \otimes V^* \to \bigoplus_{V^* \in \widehat{G}_{K,U^*}} (V \otimes U^*)^K \otimes V^*  \; . 
\]

In particular, if $\lambda_V$ is the Casimir eigenvalue of $\Delta_g^V$, then the $\lambda_V$-eigenspace $\Isot(V^*)_{\lambda_V}$ of the restricted operator $\Delta_{g,U^*}|_{\Isot(V^*)_{\lambda_V}}$ satisfies
\[
\Isot(V^*)_{\lambda_V} \simeq (V \otimes U^*)^K\otimes V^* \simeq (V^*)^{\oplus \dim_{\C} (V\otimes U^*)^K } \, .
\]
\end{theorem}

\begin{proof}
Let $(v\otimes \omega)\otimes \xi \in (V\otimes U^*)\otimes V^* \simeq \Isot(V^*)$ and recall the isomorphism \ref{eq:isomorphisms-isotypical-general-decomposition}. For any $x \in G$,
\[
    \begin{array}{ccl}
       \Delta_{g,U^*} \, \phi_{v,\omega,\xi}\,(x)  &=& -\sum\limits_j  [L_*(Y_j)^2 \, \phi_{v,\omega,\xi}](x)  \\
         &=& -\sum\limits_j \left. \frac{d^2}{dtds}\right|_{s,t=0} \phi_{v,\omega,\xi} (\, \exp(-(t+s)Y_j)x \,) \\
         &=& -\sum\limits_j \left. \frac{d^2}{dtds}\right|_{s,t=0} [\,\xi(\, \Pi(\exp(-(t+s)Y_j)) \Pi(x)v \,)\,]\omega  \\
         &=& [\,\xi(\, \Delta_g^V(\Pi(x)v) \,)\,]\omega  \\
         &=& [\,\xi \circ \Delta_g^V(\, \Pi(x)v \,)\,]\omega\\
         &=& [\,(\Delta_g^{V^*} \xi)(\, \Pi(x)v \,)\,]\omega  \\
          &=& [\, \Pi^*(x)^{-1} \, (\Delta_g^{V^*}\xi)\, (v) \,]\omega \\
          &=& \phi_{v,\omega, (\Delta_g^{V^*}\xi)}\,(x) \, .
    \end{array}
\]
This proves the first identification.

Now, we know that $\Delta_g^V$ is a multiple of the identity and it commutes with $\Pi(x)$ for all $x \in G$. Also, we know that $\xi \circ \Delta_g^V = \Delta_g^{V^*} \xi$. Therefore,
       \[
    \begin{array}{ccl}
      \phi_{(\Delta_g^V v), \omega, \xi}\,(x) &=& [\, \xi (\Pi(x)\Delta_g^V v)\,]\omega  \\
       &=& [\, \xi (\Delta_g^V \Pi(x)v)\,]\omega  \\ &=& [\, \xi \circ \Delta_g^V ( \Pi(x)v)\,]\omega  \\  
       &=& [\, \Delta_g^{V^*}\xi ( \Pi(x)v)\,]\omega  \\ 
       &=&  \phi_{v,\omega, (\Delta_g^{V^*}\xi)} (x)  \\  
       &=&   \Delta_{g,U^*} \, \phi_{v,\omega,\xi}\,(x) \, . \\ 
    \end{array}
    \]
This proves the second identification.
\end{proof}

Before we continue, we need to study some algebraic properties on Casimir elements and the root systems.

\,

\,

\subsection{Representations with the same Casimir eigenvalue} \label{subsec:Root-systems-and-representations-with-the-same-Casimir-eigenvalue}

\,

\,
Let $(G,K)$ a pair of Lie groups such that $M=G/K$ is a compact and connected normal homogeneous space. The manifold $M$ is endowed with a normal metric $g$ corresponding to a bi-invariant metric on $G$.

Petrecca and Röser (2018, \cite{Petrecca2019}) had established that the Laplace-Beltrami operator on symmetric spaces with rank $\geq 2$ do not satisfied the irreducibility criterium. Thus, a generic panorama is still an open problem for these spaces. There is a reason behind it and the argument is essentially the same for normal homogeneous spaces and operators in the form $\Delta_{g,U^*}$.
We will show that for a normal homogeneous space, there are some algebraic structural symmetries for the eigenspaces of $\Delta_{g,U^*}$ which relate irreducible representations with the same Casimir eigenvalue. These structural symmetries are implicitly encoded by the Freudenthal's formula \ref{eq:Freudenthal-formula} and basic properties of the root systems.  So, after we take in account all the geometric and structural symmetries, we will be able to show a generic situation for $\Delta_{g,U^*}$.

\,

\begin{remark} \label{remark:representations-with-the-same-casimir-eigenvalue}
 By Remark \ref{remark:same-casimir-eigenvalue}, we note that all the representations in $\widehat{G}$ with Casimir eigenvalue $a^2-g(\delta,\delta)$ are completely described by the set
\[
S(a):= \{ \mu \in  \Gamma_G^* \; | \; V^\mu \in \widehat{G} \; \text{e} \; a_\mu = a  \} = \S_a(-\delta) \cap  \Gamma_G^* \; .
\]
Even the elements $\mu \in S(a)-C$ describes some representation $V^{w_\mu \cdot \mu} \in \widehat{G}$, where $w_\mu$ is an element of the Weyl group $W$ which sends $\mu$ to the Weyl chamber $C$.
\end{remark}

\,

Due to the sphere $\S_a(-\delta)$, which contains the highest weights associated to the same Casimir eigenvalue as analyzed in Remarks \ref{remark:same-casimir-eigenvalue} and \ref{remark:representations-with-the-same-casimir-eigenvalue}, we want to change the origin from $0$ to $-\delta$ in our root system. That is, we want to put a copy of the root system on a new origin $-\delta$ in a way which is similar to the process of putting a tangent space on an arbitrary point of a smooth manifold.  We will do that in the next subsection.

\,

\,

 \subsection{The linear structure on $-\delta$} \label{subsec:The-linear-structure-on--delta}

Let $\widetilde{\a}$ equals to $\a$ as sets and consider the vector operations of sum, scalar multiplication and inner product $\big(\; +\, , \odot\, , \, \llangle \cdot \, , \, \cdot \, , \, \rrangle \; \big)$ over $\widetilde{\a}$, respectively, defined by
 \[
\forall H,H' \in \widetilde{\a} , \, \forall c \in \R, \;  \left\{ \begin{array}{l}
    (H-\delta) \oplus (H'-\delta) := (H-H')-\delta  \, , \\
    c \odot (H-\delta) :=   (cH)-\delta \, , \\
    \llangle H-\delta , H'-\delta \rrangle := \langle H, H' \rangle \, .
\end{array} \right.
\]

\begin{remark} \label{remark:linear-isometry-delta-deslocada}
    Let $T_{-\delta}:\a \to \widetilde{\a}$ the translation map $H \mapsto H-\delta$. One can easily check that $T_{-\delta}$ is a linear isometry with inverse $T_{\delta}:\widetilde{\a} \to \a$.
\end{remark}

Remark \ref{remark:linear-isometry-delta-deslocada} give us a way to move for $\a$ to $\widetilde{\a}$ (and vice-versa). There are other induced relations between these vector spaces. For example, every $\varphi \in \End(\a)$ induces a map
\[
\widetilde{\varphi}: \widetilde{\a}\ni (H-\delta) \mapsto (\varphi(H)-\delta) \in \tilde{\a} 
\]
If we set 
\[
\widetilde{\End(\a)}:= \{ \widetilde{\varphi} \, ; \; \varphi \in \End(\a) \} \; , 
\]
then $\widetilde{\End(\a)} = \End(\widetilde{\a})$ and this ring is isomorphic to the ring $\End(\a)$. Similarly, if we set 
\[
\widetilde{O(\a)}:= \{ \widetilde{\varphi} \, ; \; \varphi \in O(\a) \} \; , 
\]
then $\widetilde{O(\a)} = O(\widetilde{\a})$.

Recall Remark \ref{remark:representations-with-the-same-casimir-eigenvalue} and define
\[
O(\widetilde{\a})_{S(a)}:= \{ \widetilde{\varphi} \in O(\widetilde{\a}) ; \; \widetilde{\varphi}(S(a)) \subset S(a) \} \; .
\]
Since $S(a)$ is a finite set and each map in $O(\widetilde{\a})$ is one-to-one, then we have the equality

\begin{equation*}
    O(\widetilde{\a})_{S(a)} = \{ \widetilde{\varphi} \in O(\widetilde{\a}) ; \; \widetilde{\varphi}(S(a)) = S(a) \} \; 
\end{equation*}
and $O(\widetilde{\a})_{S(a)}$ is a subgroup of $O(\widetilde{\a})$.

\,

\,

 \subsubsection{ The action of combinatorial groups of hidden symmetries in the root system} \label{subsubsec:structure-on--delta}
 
\,

\,

\begin{theorem} \label{th:transitive-action}
 Consider the precedent constructions for the root system associated to the pair $(G,K)$, and let $a>0$ such that $\lambda := a^2 - (\delta, \delta)$ is a Casimir eigenvalue. The group  $O(\widetilde{\a})_{S(a)}$ is a finite subgroup of $O(\widetilde{\a})$ and it acts transitively on $S(a)$. 
 \end{theorem}

 Before we show the proof of the theorem, we do some remarks and a lemma.

 The dislocated orthogonal transformations in the group $O(\widetilde{\a})_{S(a)}$ can be considered hidden symmetries for the spectral problem in the sense that they preserve the eigenspace associated to the Casimir eigenvalue  $\lambda = a^2 - g(\delta, \delta)$.

 We recall that the Weyl group of $\a$ is denoted by $W$ and, in our context, we are assuming $\delta \in \Gamma_G^*$ (for example, when $\Gamma_G = \I$ as in the simply connected case) and $\a$ irreducible. We state a lemma before the proof of the precedent theorem.

\begin{lemma} \label{lemma:weyl-group}
   Let $\widetilde{W}:= \{ \widetilde{w} \in O(\widetilde{\a}) \, ; \;  w \in W \}$. Then $\widetilde{W} \subset O(\widetilde{\a})_{S(a)}$.
\end{lemma}

\begin{proof}
    Let $\mu \in S(a) = \S_a(-\delta) \cap \Gamma_G^*$. Then for each $w \in W$, we have
    \[
    \begin{array}{rcl}
      \widetilde{w}(\mu)&=&\widetilde{w}(\mu+\delta-\delta)  \\
                        &=&w(\mu+\delta)-\delta  \\
                        &=&w(\mu)+w(\delta)-\delta \, . 
    \end{array}
    \]
    Since $W(\Gamma_G^*) \subset \Gamma_G^*$ and $\Gamma_G^*$ is a lattice, then $w(\mu)+w(\delta)-\delta \in \Gamma_G^*$. Also, since $\widetilde{w} \in O(\widetilde{\a})$, we have $\widetilde{w}(\mu) \in \S_a(-\delta)$. Thus, we conclude that $\widetilde{w}(S(a)) \subset S(a)$, which proves the lemma. 
\end{proof}

\begin{remark} \label{remark:W-irreducible-action}
    It is well known that, when $\a$ is a irreducible root system, the $W$-action on $\a$ is irreducible. Thus, the $\widetilde{W}$-action on $\widetilde{\a}$ is also irreducible. 
\end{remark}

\begin{proof}[Proof of Theorem \ref{th:transitive-action}]
Choose $\mu \in S(a) = \S_a(-\delta) \cap \Gamma_G^*$. Since $a>0$, we have $\mu + \delta \neq 0$. By Remark \ref{remark:W-irreducible-action}, the set $\widetilde{W}\cdot \mu$ spans the vector space $\widetilde{\a}$. Thus, we can choose elements $ w_1, \dots, w_r \in W$, $w_1 = \id$, such that $\beta:= \{ \widetilde{w}_1 \cdot \mu , \dots, \widetilde{w}_r \cdot \mu \}$ is a basis of $\widetilde{\a}$. By Lemma \ref{lemma:weyl-group}, $\beta \subset S(a)$.

Let $\widetilde{\varphi} \in O(\widetilde{\a})_{S(a)}$. Since $\widetilde{\varphi}$ is completely characterized by its values evaluated on the basis $\beta$ and $S(a)$ is a finite set, then we have only finite possibilities to define $\widetilde{\varphi}$. Thus, $O(\widetilde{\a})_{S(a)}$ must be a finite set.

Now, let $\eta \in S(a)$. By the same argument in the last paragraph, we can construct $\widetilde{\varphi}\in O(\widetilde{\a})_{S(a)}$ such that $\widetilde{\varphi}(\widetilde{w_1} \cdot \mu) 
 = \widetilde{\varphi}(\mu) = \eta$, which proves that the action is transitive.
\end{proof} 

\,

\begin{remark} \label{remark:G_K,U*}
    Note that, the constructions in this subsection apply to the root system of $\g$ endowed with an inner product provided by a bi-invariant metric $g$ on a Lie group $G$ with $\Lie(G) = \g$. Therefore the analysis on representations with same Casimir element applies to the set $\widehat{G}$ as well to the their subsets in the form $\widehat{G}_{K,U^*}$ for a $K$-representation $U^*$.
\end{remark}

\,

\subsection{Operators on normal homogeneous spaces} \label{subsec:operators-on-normal-homogeneous-spaces}

\,

\,

We want to use and to explore the consequences of the Theorem \ref{th:sec4-laplacian-representative-functions-normal-homogeneous-spaces} for the description on the eigenspaces of $\Delta_{g,U^*}$ (and its real version $\Delta_{g,U^*_{\R}}$).

First, we note that if $(V \otimes U^*)^K$ is endowed with a real inner product, then the orthogonal group $O((V \otimes U^*)^K)$ acts irreducibly on $(V \otimes U^*)^K$ as the standard representation. Therefore, for any irreducible representation $V^* \in \widehat{G}_{K,U^*}$, the product group $O((V \otimes U^*)^K) \times G$ acts irreducibly on the isotypical component $(V \otimes U^*)^K \otimes V^*$ of $L^2(G,K;U^*)$. 
If we define
\[
\mathbb{O}_{\C} := \prod_{V^* \in \widehat{G}_{K,U^*}} O((V \otimes U^*)^K) \, ,
\]
then these $(O_{\C}((V \otimes U^*)^K) \times G)$-actions can be extended to a unique $(\mathbb{O}_{\C} \times G)$-action on $L^2(G,K,U^*)$, in an obvious way. Thus, in this construction, the group $\mathbb{O}_{\C} \times G$ acts in a given isotypical component $(V\otimes U^*)^K\otimes V^*$ as the group $O((V \otimes U^*)^K) \times G$, effectively. Therefore, each isotypical component $(V\otimes U^*)^K\otimes V^*$ is an irreducible $(\mathbb{O}_{\C} \times G)$-submodule of $L^2(G,K,U^*)$.

Now, in order to establish the criterium of irreducibility for the eigenspaces of $\Delta_{g,U^*}$, we only have to determine when two distinct isotypical components of $L^2(G,K,U^*)$ produce the same eigenvalue or not. For those with the same eigenvalue, we also have to determine if they are related by a structural symmetry or not, in some way. We do that in the following theorem. Before we enunciate it, we recall Remarks \ref{remark:representations-with-the-same-casimir-eigenvalue} and \ref{remark:G_K,U*} and we define
\[
S(a;U^*)  := \{ \mu \in  C \cap \Gamma_G^* \; | \; V^\mu \in \widehat{G}_{K,U^*} \; \text{e} \; a_\mu = a  \} \; ,
\]
which is a subset of $S(a) \cap C = \S_a(-\delta) \cap  \Gamma_G^* \cap C$, for each $a>0$.

\,

\begin{theorem} \label{th:sec4-esps-homog-normais-mesmo-autovalor-casimir} Let $M=G/K$ a compact and connected homogeneous space, $g$ a normal metric on $M$ corresponding to a bi-invariant product on $G$ and to a $Ad$-invariant inner product on $\g$, both of then also denoted by $g$ for the sake of simplicity, and $U^*$ a $K$-module. Let $\delta$ the half sum of positive roots associated to a fixed Weyl Chamber for the root system of $(G,K)$,  $\lambda(a) = a^2 - g(\delta,\delta)$ an eigenvalue of the generalized Laplace operator $\Delta_{g,U^*}$ for some $a \geq 0$, and $S(a,U^*)$ as defined above. Then, the $\lambda(a)$-eigenspace $L^2(G,K;U^*)_{\lambda(a)}$ of $\Delta_{g,U^*}$ satisfies 
\[
L^2(G,K;U^*)_{\lambda(a)} \simeq \bigoplus_{\mu \in S(a;U^*)} (V^\mu \otimes U^*)^K \otimes V^{\mu^*} 
\]
and the elements of $S(a;U^*)$ are related by a transitive action of a finite subgroup of $\Iso (\S_{a}(-\delta))$ in the root system of $\g$. Thus, each eigenspace of $\Delta_{g,U^*}$ is an irreducible $(\mathbb{O}_{\C} \times G)$-module or it is a finite sum of irreducible $(\mathbb{O}_{\C} \times G)$-modules which are related by finite isometries of the sphere $\S_{a}(-\delta)$, in the root system of $\g$.

In particular, if $S(a,U^*)$ is contained in a root system of $rank$ $1$ for each eigenvalue $\lambda(a)$, then $S(a,U^*)$ is an unitary set and $\Delta_{g,U^*}$ is complex $(\mathbb{O}_{\C} \times G)$-simple.
\end{theorem}

\begin{proof}
By the Freudenthal's formula \ref{eq:Freudenthal-formula} and the basic properties of the root systems, we know that all representations in $\widehat{G}_{K,U^*}$ with Casimir eigenvalue $\lambda(a)$ are completely characterized by the set $S(a;U^*)$ in such way that  
\[
L^2(G,K;U^*)_{\lambda(a)} \simeq \bigoplus_{\mu \in S(a;U^*)} (V^\mu \otimes U^*)^K \otimes V^{\mu^*} \; .
\]
By the previous paragraphs, each isotypical component $(V^\mu \otimes U^*)^K \otimes V^{\mu^*}$ is an irreducible $(\mathbb{O}_{\C} \times G)$-module. Putting these constructions with the content of Subsection \ref{subsec:Root-systems-and-representations-with-the-same-Casimir-eigenvalue} we see that the remaining properties of this theorem is merely a consequence of Theorem \ref{th:transitive-action}.
\end{proof}

\begin{remark}
    Why  do we not consider a $Q_8$-action in Theorem \ref{th:sec4-esps-homog-normais-mesmo-autovalor-casimir}? The answer is divided in two parts. Firstly, in practical terms, the $Q_8$-group developed on Subsection \ref{subsec:real-vs-complex} does not play any role for symmetric spaces of $\rank$ $1$, due to the fact that we are considering a $\mathbb{O}_{\C}$-action on $(V^\mu\otimes U^*)^K$ instead of a $G$-action (independently of the type of the $G$-representation $V^\mu$) and due to the fact that there are not exist representations of complex type in $\rank$ $1$. The second part of the answer is related to the properties of root systems with $\rank \geq 2$. In fact, in this case one could consider a $Q_8$-action, but it would irrelevant because that there can be exist nonequivalent and non-dual representations with the same Casimir eigenvalue, in such way that $Q_8$ can not organize all of them ($Q_8$ is too small for such a purpose, while the $(\mathbb{O}_{\C} \times G)$-action and the isometries of the spheres centered in $-\delta$ in the root system are more suitable for the desired descriptions).  
\end{remark}

\,

\begin{remark} \label{remark:sec4-focando-no-operador-complexo}
    For convenience of the representation theory over $\C$, we turned back our attention to the operator $\Delta_{g,U^*}$, forgetting $\Delta_{g,U^*_{\R}}$ in a first moment. The analogous of the complex version $\Delta_{g,U^*}$ for its real version $\Delta_{g,U^*_{\R}}$ is essentially achieved by considering real forms for the suitable complex spaces in our constructions. Now, only this time, we will do that passage in order to illustrate this process. After that, we will only study $\Delta_{g,U^*}$. 
\end{remark}

\,

In order to establish an analogous of the Theorem \ref{th:sec4-esps-homog-normais-mesmo-autovalor-casimir} for the real operator $\Delta_{g,U^*_{\R}}$ we must consider some adaptations:

\,

\begin{itemize}
    \item[(i)] Take $\mathcal{R}[(V \otimes U^*)^K]$ as a real form for the space $(V \otimes U^*)^K$.

    \item[(ii)]  Define $[V] := \{ V, V^* \}$ and $\widehat{\mathcal{G}}:= \{ \;  [V] \; , V \in \widehat{G}_{K,U^*} \}$.

    \item[(iii)] Take $V_{\R}$ a irreducible real representation such that  
        \begin{equation*}
         \C \otimes V_{\R} \simeq \left\{ \begin{array}{l}
        \text{$V$, if $V$ is of real type,} \\
        \text{$\H\otimes V$, if $V$ is of complex or quaternionic type.}
        \end{array} \right. 
        \end{equation*}

    \item[(iv)] Note that $C^\infty(G,K;U^*_{\R}) = C^\infty(G,U^*) \cap C^\infty(G,K;U^*_{\R})$ and, in the previous notations,
    \[
    L^2(G,K;U^*_{\R})   \simeq \bigoplus_{[V] \in \widehat{\mathcal{G}}} \mathcal{R}[(V \otimes U^*)^K] \otimes V_{\R}^* \; .
    \] 

    \item[(v)] Define    
    \[
    \mathbb{O}_{\R} := \prod_{[V] \in \widehat{\mathcal{G}}} O(\mathcal{R}[(V \otimes U^*)^K]) \; 
    \]
    and consider the $(\mathbb{O}_{\R} \times G)$-action on 
    \[
    L^2(G,K;U^*_{\R}) \simeq \bigoplus_{[V] \in \widehat{\mathcal{G}}} \mathcal{R}[(V \otimes U^*)^K] \otimes V_{\R}^* \; . 
    \]

    \item[(vi)] Define $[\mu]:= \{\mu, \mu^* \}$ and $[S(a,U^*)] := \{  [\mu] \, ; \; \mu \in S(a,U^*)   \} $.
\end{itemize}

\,

With those adaptations, we conclude the following theorem for the real version $\Delta_{g,U^*_{\R}}$ of the operator $\Delta_{g,U^*}$:

\begin{theorem}  \label{th:sec4-esps-homog-normais-mesmo-autovalor-casimir-caso-real} Let $M=G/K$ a compact and connected homogeneous space, $g$ a normal metric on $M$ corresponding to a bi-invariant product on $G$ and to a $Ad$-invariant inner product on $\g$, both of then also denoted by $g$ for the sake of simplicity, and $U^*$ a $K$-module. Let $\delta$ the half sum of positive roots associated to a fixed Weyl Chamber for the root system of $(G,K)$, and $\lambda(a) = a^2 - g(\delta,\delta)$ an eigenvalue of the generalized Laplace operator $\Delta_{g,U^*}$ for some $a \geq 0$. Then, the $\lambda(a)$-eigenspace $L^2(G,K;U^*_{\R})_{\lambda(a)}$ of the real operator $\Delta_{g,U^*_{\R}}$ satisfies 
\[
L^2(G,K;U^*_{\R})_{\lambda(a)} \simeq \bigoplus_{[\mu] \in [S(a;U^*)]} \mathcal{R}[(V^\mu \otimes U^*)^K] \otimes V_{\R}^{\mu^*}
\]
and the elements of $[S(a;U^*)]$ are related by a transitive action of a finite subgroup of $\Iso (\S_{a}(-\delta))$, in the root system of $\g$. Thus, each eigenspace of $\Delta_{g,U^*_{\R}}$ is an irreducible $(\mathbb{O}_{\R} \times G)$-module or it is a finite sum of irreducible $(\mathbb{O}_{\R} \times G)$-modules which are related by finite isometries of the sphere $\S_{a}(-\delta)$, in the root system of $\g$.

In particular, if $S(a,U^*)$ is contained in a root system of $rank$ $1$, for each eigenvalue $\lambda(a)$, then $S(a,U^*)$ is an unitary set and $\Delta_{g,U^*_{\R}}$ is complex $(\mathbb{O}_{\R} \times G)$-simple.
\end{theorem}

\,

\begin{example} \label{ex:hodge-total-normal-homogeneous-spaces}
   Let $\Delta_g^H$ the Hodge-Laplace operator $\Delta_g^H$ on the total space of the complex differential forms, on a compact and connected normal homogeneous space $M = G/K$ with normal metric $g$. If we take $U^* := \bigoplus_p \wedge^p \m^{*\C}$, then $\Delta_g^H$ can be identified as the operator 
        \[
        \Delta_g^{H} \simeq  \Delta_{g,U^*} \; : \; \bigoplus\limits_{p} C^\infty(G,K;\wedge^p \m^{*\C}) \to \bigoplus\limits_{p} C^\infty(G,K;\wedge^p \m^{*\C}) \, .
        \]
    Thus, $\Delta_g^H$ satisfies Theorem \ref{th:sec4-esps-homog-normais-mesmo-autovalor-casimir}.   
    Similarly, if we consider the restriction of $\Delta_g^H$ to the real forms and we take $U^*_{\R} := \bigoplus_p \wedge^p \m^{*}$, then we can conclude that it satisfies Theorem \ref{th:sec4-esps-homog-normais-mesmo-autovalor-casimir-caso-real}.
\end{example}

In some cases, Example \ref{ex:hodge-total-normal-homogeneous-spaces} can be more detailed, as we will see in Example \ref{ex:hodge-total-rank1} for the Lie group $M=G = (G\times G)/\Delta G$ with $\rank(G) = 1$. We present two lemmas before it. 

\,

\begin{lemma} \small \label{lema:ikeda-fechadas-cofechadas} 
Let $E_\lambda^p$ the $\lambda$-eigenspace of $\Delta_g^H$ restricted to the complex differential $p$-forms on a normal homogeneous space $M=G/K$ with metric $g$. Define
\[
E_\lambda^{p,f} := \{ \omega \in E_\lambda^p \; | \; \text{$\omega$ is closed} \} \; , \quad E_\lambda^{p,cf} := \{ \omega \in E_\lambda^p \; | \; \text{$\omega$ is co-closed} \} \; . 
\]
Then 
\[
E_\lambda^p = \left\{ \begin{array}{c}
     \text{$E_\lambda^{p,f} \oplus E_\lambda^{p,cf} $, if $\lambda > 0 $}  \\
     \text{$E_\lambda^{p,f} = E_\lambda^{p,cf} $, if $\lambda = 0 $}
\end{array} \right. \quad .
\]
Moreover, the following maps are $G$-isomorphisms, for $\lambda>0$:
\[
d: E_\lambda^{p,cf} \to E_\lambda^{p+1,f} \;,  \quad d^*: E_\lambda^{p+1,f} \to E_\lambda^{p,cf} \quad \text{e} \quad \star:E_\lambda^{p,f} \to E_\lambda^{m-p,cf} \; .
\]
\end{lemma}

\,

The precedent lemma can be found on \cite[Sec.1]{Ikeda1978}.

\,

\begin{lemma} \label{lema:rank1}
    Let $M = (G' \times G')/\Delta G' $ a compact and connected Lie group of $\rank$ 1 with bi-invariant metric $g$. Define $G:= G' \times G'$, $K:=\Delta G'$ and $U^* := \bigoplus_p \wedge^p \m^{*\C}$. Then, 
        \[
        \widehat{G}_{K,\wedge^p\m^{*\C}} \; = \; \widehat{G}_{K} \; , \quad p =  0,1,\dots,\dim M \; .
        \]
    In particular, for each eigenvalue $\lambda(a) = a^2-g(\delta,\delta)$ of the Hodge-Laplace operator $\Delta_g^H \simeq \Delta_{g,U^*}$, we conclude that
    \[
    S(a;U^*) = S(a;\wedge^p \m^{*\C}) = S(a; \C)\; , \quad p =  0,1,\dots,\dim M \; .
    \]
\end{lemma}

\begin{proof}
    Let $V^{\mu^*} \in \widehat{G}_{K,\wedge^p(\m^*)^{\C}}$ a $G$-submodule of a eigenspace $E_\lambda^p$ of $\Delta_g^H$ restricted to the complex differential $p$-forms. Then, by Theorem \ref{th:sec4-esps-homog-normais-mesmo-autovalor-casimir},
    \[
    E_\lambda^p \simeq (V^\mu \otimes \wedge^p\m^{*\C})^K \otimes V^{\mu^*} \; .
    \]
    If we see $E_\lambda^p$ as a $G$-module, then each of its submodules are isomorphic to $V^{\mu^*}$.
    In particular, each irreducible $G$-submodule of $E_\lambda^{p,cf}$ must be isomorphic to $V^{\mu^*}$. By Lemma \ref{lema:ikeda-fechadas-cofechadas}, $E_\lambda^{p+1,f}$ has a $G$-submodule isomorphic to $V^{\mu^*}$. Therefore, $V^{\mu^*} \in \widehat{G}_{K,\wedge^{p+1}\m^{*\C}}$. So, we conclude that
    \[
    \widehat{G}_{K} = \widehat{G}_{K,\wedge^{0}\m^{*\C}} \subset \widehat{G}_{K,\wedge^{1}\m^{*\C}} \subset \cdots \subset \widehat{G}_{K,\wedge^{\dim M}\m^{*\C}} = \widehat{G}_{K} \; .
    \]
\end{proof}

\begin{example} \label{ex:hodge-total-rank1}
   Let $\Delta_g^H$ the Hodge-Laplace operator $\Delta_g^H$ on the total space of the complex differential forms, on a compact Lie group $M = G$ with $\rank(G)= 1$ and bi-invariant metric $g$. Set $U^* := \bigoplus_p \wedge^p \g^{*\C}$. Then, by Theorem \ref{th:sec4-esps-homog-normais-mesmo-autovalor-casimir} and Lemma \ref{lema:rank1}, the operator 
        \[
        \Delta_g^{H} \simeq  \Delta_{g,U^*} \; : \; \bigoplus\limits_{p} C^\infty(G;\wedge^p \g^{*\C}) \to \bigoplus\limits_{p} C^\infty(G;\wedge^p \g^{*\C}) 
        \]
    has a complex $(\mathbb{O}_{\C} \times G)$-simple spectrum.   
    Similarly, if we add Theorem \ref{th:sec4-esps-homog-normais-mesmo-autovalor-casimir-caso-real} in our analysis with $U^*_{\R} := \bigoplus_p \wedge^p \g^{*}$, then we conclude that the real version of $\Delta_g^{H}$ has a real $(\mathbb{O}_{\R} \times G)$-simple spectrum.
    We can go beyond that. Note that each eigenspace of the complex operator $\Delta_g^H$ has the form
    \[
    (V \otimes U^*) \otimes V^* \simeq U^* \otimes (V \otimes V^*) \, .
    \]
    Therefore, $\Delta_g^{H}$ has a complex $(O(U^*) \times (G\times G))$-simple spectrum. Similarly, the real version of $\Delta_g^H$ has a real $(O(U^*_{\R}) \times (G\times G))$-simple spectrum. In particular, this construction applies to groups as $SU(2)$ or $SO(3)$, for example.
\end{example}

\,

\section{Generic estimate for the eigenspaces on Lie groups with left-invariant metrics} \label{sec:back-to-Lie-groups-left-metrics}

\,

\,

In this section, we recommend the reader to recall Sections \ref{sec:generalized-laplacians} and \ref{sec:generalized-laplacians-normal-homogeneous-spaces}, since we will use definitions, constructions and results presented in both of them.

Let $G$ a Lie group endowed with a left-invariant metric $g$. Take $\{Y_j\}$ an any $g$-orthonormal basis of $\g$.

\,

\begin{remark} \label{remark:Lie-groups-bi-invariant-metrics}
    We recall, from Section \ref{sec:generalized-laplacians}, that the operator $\Delta_{g,U^*}$ was given by
    \[
    \Delta_{g,U^*}:= -\sum\limits_{j} R_*(Y_j)^2 :C^\infty(G,U^*) \to C^\infty(G;U^*) \; .
    \] 
    By convenience, for the left-invariant metrics $g$ which are also bi-invariant metrics, we will consider the operator $\Delta_{g,U^*}$ as
    \[
    \Delta_{g,U^*}:= -\sum\limits_{j} L_*(Y_j)^2 :C^\infty(G,U^*) \to C^\infty(G;U^*) \; .
    \] 
    There is not any problem to do that once: 
        \begin{enumerate}
            \item[(i)] $R$ and $L$ define equivalent representations on $C^\infty(G,U^*)$ (this equivalence is induced by the inversion map on $G$),

            \item[(ii)] for every bi-invariant metric $g$, the operators $\sum_j R_*(Y_j)^2$ and $\sum_j L_*(Y_j)^2$ are simply the Casimir elements of these equivalent representations, respectively.
        \end{enumerate}  
    Therefore, $\sum_j R_*(Y_j)^2$ and $\sum_j L_*(Y_j)^2$ are essentially the same operator for every bi-invariant metric $g$. The main advantage achieved by this convention is that we can use not only the results on Section \ref{sec:generalized-laplacians}, but we can also enjoy the content in Subsection \ref{subsec:operators-on-normal-homogeneous-spaces} for the operators $\Delta_{g,U^*}$, where $g$ is a bi-invariant metric.    
\end{remark}

\,

We begin presenting an improvement of Theorem \ref{teo:cap2-grupos-de-Lie-metricas-invariantes-a-esquerda}. Once we do that, we will be able to present an estimate for the eigenspaces of $\Delta_{g,U^*}$, where $g$ is a generic left-invariant metric on $G$.

\,

\begin{theorem} \label{teo:cap2-grupos-de-Lie-metricas-invariantes-a-esquerda-aprimorado}
    A left-invariant metric $g$ on a compact and connected Lie group $G$ satisfies that the operator $\Delta_{g,U^*_{\R}}$ has a real $(O(U^*_{\R}) \times G)$-simple spectrum if and only if the operator $\Delta_{g,U^*}$ has a complex $(O(U^*) \times \widetilde{G})$-simple spectrum. Beyond that, the existence of a such metric is equivalent to following items simultaneously satisfied:

         \begin{enumerate}
          \item Fix a bi-invariant metric $g_0$ corresponding to a $Ad$-invariant inner product in $\g$ also denoted by $g_0$ for convenience, take the corresponding root system on $\g$, and assume that $C$ is a chosen Weyl chamber and $\delta \in C$ is the half sum of positive roots.  Then for each eigenvalue of the operator $\Delta_{g_0,U^*}$, in the form $\lambda(a) = a^2 - g_0(\delta,\delta)$ with $a\geq 0$, and for all $\mu,\eta \in S(a;U^*)$ with $\mu \neq \eta, \eta^*$, the polynomial $a_{V^\mu,V^\eta}$ does not vanish identically.

          \item For all $V\in \widehat{G}$ of real or complex type, $b_{V}$ does not vanish identically.

          \item For all $V\in \widehat{G}$, of quaternionic type, $c_{V}$ does not vanish identically.
      \end{enumerate}
Moreover, such metric, when its existence holds, is generic on the space of the left-invariant metrics of $G$.   
\end{theorem}

\,

\begin{proof} \small
    This theorem differs from Theorem \ref{teo:cap2-grupos-de-Lie-metricas-invariantes-a-esquerda} only by the item (1). By comparing both of them, we see that the only verification remaining to do is the following statement:
        \begin{itemize}
            \item[(S)] \emph{Let $\lambda(a) = a^2 - g_0(\delta,\delta)$ a eigenvalue for the operator $\Delta_{g_0,U^*}$ and consider  $V^\mu$ and $V^{\mu'}$ a pair of irreducible representations such that $\mu \in S(a;U^*)$ and $\mu' \in S(a';U^*)$. If $a \neq a'$, then $a_{V^\mu,V^{\mu'}}(g_0) \neq 0$.}
        \end{itemize}
     In fact, if $a \neq a'$ then $\Delta_{g_0}^{V^\mu}$ and $\Delta_{g_0}^{V^{\mu'}}$ have distinct Casimir eigenvalues $\lambda(a)$ e $\lambda(a')$, respectively. Therefore, $a_{V^\mu,V^{\mu'}}(g_0) \neq 0$.
\end{proof}

\,

Note that there is much less representations to be tested in the item (1) of Theorem \ref{teo:cap2-grupos-de-Lie-metricas-invariantes-a-esquerda-aprimorado} compared to the item (1) of Theorem \ref{teo:cap2-grupos-de-Lie-metricas-invariantes-a-esquerda}. As a consequence, for a generic left-invariant metric on a compact and connected Lie group $G$, we have the following generic estimate for the eigenspaces:

\,

\,

\begin{corollary} \label{cor:cap2-grupos-de-Lie-metricas-invariantes-a-esquerda-aprimorado}
Fix a bi-invariant metric $g_0$ on a compact and connected Lie group $G$. Take the corresponding $Ad$-invariant inner product and the corresponding root system on $\g$, with a chosen Weyl chamber $C$ and half sum of positive roots $\delta \in C$. Denote by $E(\lambda,g)$ the $\lambda$-eigenspace of the generalized Laplacian $\Delta_{g,U^*}$. Suppose that $E(\lambda,g)$ contains an irreducible $G$-submodule $V^{\mu_\lambda}$ and define the scalar $a_\lambda := g_0(\mu_\lambda +\delta,\mu_\lambda+\delta)^{1/2} $. Then, a generic left-invariant metric $g$ satisfies
\[
E(\lambda, g) \;\; \leq \bigoplus_{\mu \in S(a_\lambda ;U^*)} (V^{\mu^*})^{\oplus \dim_{\C} (V^\mu \otimes U^*)^K} \, .
\]
Similarly, in the same notations of Theorem \ref{th:sec4-esps-homog-normais-mesmo-autovalor-casimir-caso-real}, we can conclude that, if $E_{\R}(\lambda,g)$ is the $\lambda$-eigenspace of the real operator $\Delta_{g,U^*_{\R}}$, then a generic left invariant metric $g$ satisfies
\[
E_{\R}(\lambda, g) \;\; \leq \bigoplus_{[\mu] \in [S(a_\lambda ;U^*)]} (V_{\R}^{\mu^*})^{\oplus \dim_{\C} (V^\mu \otimes U^*)^K} \; .
\]
\end{corollary}

\begin{proof}
 We observe that $\dim_{\C} (V^\mu \otimes U^*)^K = \dim_{\R} \mathcal{R}[(V^\mu \otimes U^*)^K]$. In order to achieve this corollary, we only have to combine Theorems \ref{th:sec4-esps-homog-normais-mesmo-autovalor-casimir}, \ref{th:sec4-esps-homog-normais-mesmo-autovalor-casimir-caso-real} and \ref{teo:cap2-grupos-de-Lie-metricas-invariantes-a-esquerda-aprimorado}.
\end{proof}

\,


\subsection*{Acknowledgements}

This research was partially supported by Coordenação de Aperfeiçoamento de Pessoal de Nível Superior (CAPES) and partially supported by Fundação de Amparo à Pesquisa do Estado do Amazonas (FAPEAM).


\subsection*{Conflict of interest statement}
The authors state that there is no conflict of interest.
\subsection*{Data availability statement}
The manuscript has no associated data.

\,

\small
\bibliographystyle{apalike}

\begin{thebibliography}{}

\bibitem[Afrouzi and Rasouli, 2007]{afrouzi2007population}
Afrouzi, G. and Rasouli, S. (2007).
\newblock Population models involving the p-laplacian with indefinite weight
  and constant yield harvesting.
\newblock {\em Chaos, Solitons \& Fractals}, 31(2):404--408.

\bibitem[Arvanitoyeorgos, 2003]{Arvanitoyeorgos2003}
Arvanitoyeorgos, A. (2003).
\newblock {\em {An introduction to Lie groups and the geometry of homogeneous
  spaces}}, volume~22.
\newblock AMS.

\bibitem[Baklouti and Nomura, 2018]{baklouti2018geometric}
Baklouti, A. and Nomura, T. (2018).
\newblock {\em Geometric and Harmonic Analysis on Homogeneous Spaces and
  Applications: TJC 2015, Monastir, Tunisia, December 18-23}, volume 207.
\newblock Springer.

\bibitem[Br{\"{o}}cker and tom Dieck, 1985]{Brocker1985}
Br{\"{o}}cker, T. and tom Dieck, T. (1985).
\newblock {\em {Representations of Compact Lie Groups}}, volume~98 of {\em
  Graduate Texts in Mathematics}.
\newblock Springer Berlin Heidelberg, Berlin, Heidelberg.

\bibitem[Cahen, 1972]{cahen1972lorentzian}
Cahen, M. (1972).
\newblock Lorentzian symmetric spaces.
\newblock {\em General Relativity and Gravitation}, 3(1):115--117.

\bibitem[Cariglia, 2014]{cariglia2014hidden}
Cariglia, M. (2014).
\newblock Hidden symmetries of dynamics in classical and quantum physics.
\newblock {\em Reviews of Modern Physics}, 86(4):1283.

\bibitem[Casarino et~al., 2022]{casarino2022weighted}
Casarino, V., Ciatti, P., and Martini, A. (2022).
\newblock Weighted spectral cluster bounds and a sharp multiplier theorem for
  ultraspherical grushin operators.
\newblock {\em International Mathematics Research Notices},
  2022(12):9209--9274.

\bibitem[Caselle and Magnea, 2004]{Magnea2002introduction}
Caselle, M. and Magnea, U. (2004).
\newblock Random matrix theory and symmetric spaces.
\newblock {\em Phys. Rep.}, 394(2-3):41--156.

\bibitem[Cianci et~al., 2024]{cianci2024spectral}
Cianci, D., Judge, C., Lin, S., and Sutton, C. (2024).
\newblock Spectral multiplicity and nodal domains of torus-invariant metrics.
\newblock {\em International Mathematics Research Notices}, 2024(3):2192--2218.

\bibitem[de~Lange et~al., 2014]{de2014laplacian}
de~Lange, S.~C., de~Reus, M.~A., and van~den Heuvel, M.~P. (2014).
\newblock The laplacian spectrum of neural networks.
\newblock {\em Frontiers in computational neuroscience}, 7:189.

\bibitem[Enciso and Peralta-Salas, 2012]{enciso2012nondegeneracy}
Enciso, A. and Peralta-Salas, D. (2012).
\newblock Nondegeneracy of the eigenvalues of the hodge laplacian for generic
  metrics on 3-manifolds.
\newblock {\em Transactions of the American Mathematical Society},
  364(8):4207--4224.

\bibitem[Gier, 2014]{gier2014eigenvalue}
Gier, M.~E. (2014).
\newblock {\em Eigenvalue multiplicites of the Hodge Laplacian on coexact
  2-forms for generic metrics on 5-manifolds}.
\newblock University of Kentucky.

\bibitem[Hall, 2015]{Hall2015}
Hall, B.~C. (2015).
\newblock {\em {Lie Groups, Lie Algebras, and Representations}}, volume 222 of
  {\em Graduate Texts in Mathematics}.
\newblock Springer International Publishing, Cham.

\bibitem[Helffer, 2013]{helffer2013spectral}
Helffer, B. (2013).
\newblock {\em Spectral theory and its applications}.
\newblock Number 139. Cambridge University Press.

\bibitem[Helgason, 2001]{Helgason2001}
Helgason, S. (2001).
\newblock {\em {Differential geometry, Lie groups, and symmetric spaces}}.
\newblock AMS.

\bibitem[Helgason, 2008]{helgason2008geometric}
Helgason, S. (2008).
\newblock Geometric analysis on symmetric spaces.
\newblock {\em Mathematical surveys and monographs}.

\bibitem[Ikeda and Taniguchi, 1978]{Ikeda1978}
Ikeda, A. and Taniguchi, Y. (1978).
\newblock {Spectra and eigenforms of the laplacian on snand pn(c)}.
\newblock {\em Osaka Journal of Mathematics}, 15(3):515--546.

\bibitem[Kobayashi and Nomizu, 1963]{Kobayashi1963}
Kobayashi, S. and Nomizu, K. (1963).
\newblock {\em {Foundations of Differential Geometry Vol1}}.
\newblock Wiley.

\bibitem[Kobayashi, 2017]{Kobayashi2017}
Kobayashi, T. (2017).
\newblock {Global Analysis by Hidden Symmetry}.
\newblock {\em Progress in Mathematics}, 323:359--397.

\bibitem[Mantoiu et~al., 2016]{mantoiu2016spectral}
Mantoiu, M., Raikov, G., and de~Aldecoa, R.~T. (2016).
\newblock {\em Spectral theory and mathematical physics}, volume 254.
\newblock Springer.

\bibitem[Marrocos and Gomes, 2019]{Marrocos2019}
Marrocos, M. A.~M. and Gomes, J. N.~V. (2019).
\newblock {Generic Spectrum of Warped Products and G-Manifolds}.
\newblock {\em The Journal of Geometric Analysis}, 29(4):3124--3134.

\bibitem[Moretti, 2017]{moretti2017spectral}
Moretti, V. (2017).
\newblock Spectral theory and quantum mechanics.
\newblock {\em UNITEXT, Italy: Springer International Publishing AG}.

\bibitem[Munthe-Kaas et~al., 2001]{munthe2001application}
Munthe-Kaas, H., Quispel, G., and Zanna, A. (2001).
\newblock {\em Application of symmetric spaces and Lie triple systems in
  numerical analysis}.
\newblock Number 217. Department of Informatics, University of Bergen.

\bibitem[Petrecca and R{\"{o}}ser, 2018]{Petrecca2019}
Petrecca, D. and R{\"{o}}ser, M. (2018).
\newblock {Irreducibility of the Laplacian eigenspaces of some homogeneous
  spaces}.
\newblock {\em Mathematische Zeitschrift}, 291(1-2):395--419.

\bibitem[Ranjan, 2012]{ranjan2012discrete}
Ranjan, P. (2012).
\newblock {\em Discrete laplace operator: theory and applications}.
\newblock The Ohio State University.

\bibitem[R{\"o}ntgen et~al., 2021]{rontgen2021latent}
R{\"o}ntgen, M., Pyzh, M., Morfonios, C., Palaiodimopoulos, N., Diakonos, F.,
  and Schmelcher, P. (2021).
\newblock Latent symmetry induced degeneracies.
\newblock {\em Physical Review Letters}, 126(18):180601.

\bibitem[Schueth, 2017]{Schueth2017}
Schueth, D. (2017).
\newblock {Generic irreducibilty of Laplace eigenspaces on certain compact Lie
  groups}.
\newblock {\em Annals of Global Analysis and Geometry}, 52(2):187--200.

\bibitem[Semmelmann and Weingart, 2019]{semmelmann2019standard}
Semmelmann, U. and Weingart, G. (2019).
\newblock The standard laplace operator.
\newblock {\em manuscripta mathematica}, 158(1-2):273--293.

\bibitem[Simon, 2007]{simon2007spectral}
Simon, B. (2007).
\newblock {\em Spectral theory and mathematical physics: A festschrift in honor
  of Barry Simon's 60th birthday}.
\newblock American Mathematical Soc.

\bibitem[Singer, 2006]{singer2006linearity}
Singer, S.~F. (2006).
\newblock {\em Linearity, symmetry, and prediction in the hydrogen atom}.
\newblock Springer Science \& Business Media.

\bibitem[Tsonev, 2018]{Tsonev2018}
Tsonev, D.~M. (2018).
\newblock {On the spectrum of the hodge-laplacian acting on p-forms over
  riemannian homogeneous spaces}.

\bibitem[Uhlenbeck, 1976]{Uhlenbeck1976}
Uhlenbeck, K. (1976).
\newblock {Generic Properties of Eigenfunctions}.
\newblock {\em American Journal of Mathematics}, 98(4):1059.

\bibitem[Watson, 1973]{watson1973manifold}
Watson, B. (1973).
\newblock Manifold maps commuting with the laplacian.
\newblock {\em Journal of Differential Geometry}, 8(1):85--94.

\bibitem[Wigner, 2012]{wigner2012group}
Wigner, E. (2012).
\newblock {\em Group theory: and its application to the quantum mechanics of
  atomic spectra}, volume~5.
\newblock Elsevier.

\bibitem[Yau, 1993]{yau1993open}
Yau, S.-T. (1993).
\newblock Open problems in geometry.
\newblock In {\em Proc. Symp. Pure Math}, volume~54, pages 1--28.

\bibitem[Zachmanoglou and Thoe, 1986]{zachmanoglou1986introduction}
Zachmanoglou, E.~C. and Thoe, D.~W. (1986).
\newblock {\em Introduction to partial differential equations with
  applications}.
\newblock Courier Corporation.

\bibitem[Zelditch, 1990]{zelditch1990}
Zelditch, S. (1990).
\newblock {On the generic spectrum of a riemannian cover}.
\newblock {\em Annales de l’institut Fourier}, 40(2):407--442.

\end{thebibliography}


\end{document}